\documentclass[11pt,a4paper]{amsart}
\usepackage{amssymb,amsfonts,amsmath,mathrsfs}
\usepackage[left=2cm,right=2cm,top=2cm,height=250mm]{geometry}
\usepackage{color}
\usepackage{verbatim}
\newtheorem{proposition}{Proposition}[section]
\newtheorem{theorem}[proposition]{Theorem}
\newtheorem{corollary}[proposition]{Corollary}
\newtheorem{lemma}[proposition]{Lemma}

\theoremstyle{definition}
\newtheorem{definition}[proposition]{Definition}

\newtheorem{remark}[proposition]{Remark}

\numberwithin{equation}{section}

\def\R{\Bbb R}
\def\Dx{\Delta_x}
\def\Nx{\nabla_x}
\def\Dt{\partial_t}
\def\({\left(}
\def\){\right)}
\def\eb{\varepsilon}
\def\Cal{\mathcal}
\def\Bbb{\mathbb}

\begin{document}
\title[Quintic wave equation]{Attractors for  damped quintic wave equations in  bounded domains}
\author[] {Varga Kalantarov${}^1$, Anton Savostianov${}^2$, and Sergey Zelik${}^2$}

\begin{abstract}  The dissipative wave equation with a critical quintic nonlinearity in smooth bounded  three dimensional domain is considered. Based on the recent extension of the Strichartz estimates to the case of bounded domains, the existence of a compact global attractor for the solution semigroup of this equation is established. Moreover, the smoothness of the obtained attractor is also shown.
\end{abstract}

\thanks{
This work is partially supported by the Russian Ministry of Education and Science (contract no.
8502).}
\address{${}^1$Department of mathematics, Ko{\c c} University,
\newline\indent Rumelifeneri Yolu, Sariyer 34450\newline\indent
Sariyer, Istanbul, Turkey}
\address{${}^2$ University of Surrey, Department of Mathematics, \newline
Guildford, GU2 7XH, United Kingdom.}

\email{vkalantarov@ku.edu.tr}
\email{a.savostianov@surrey.ac.uk}
\email{s.zelik@surrey.ac.uk}
\maketitle
\tableofcontents
\section{Introduction}\label{s0}
We consider the following damped wave equation:
\begin{equation}\label{0.eqmain}
\begin{cases}
\Dt^2 u+\gamma\Dt u-\Dx u+f(u)=g,\\
u\big|_{t=0}=u_0,\ \ \Dt u\big|_{t=0}=u_0'
\end{cases}
\end{equation}
in a bounded smooth domain $\Omega$ of $\R^3$ endowed by the Dirichlet boundary conditions. Here $\gamma$ is a fixed strictly positive constant, $\Dx$ is a Laplacian with respect to the variable $x=(x^1,x^2,x^3)$, the nonlinearity $f$ is assumed to have a quintic growth rate  as $u\to \infty$:
\begin{equation}
\label{0.f5}
f(u)\sim u^5
\end{equation}
and to satisfy some natural assumptions, see Section \ref{s5} for the details, and
the initial data $\xi_u(0):=(u_0,u_0')$ is taken from  the standard energy space $\Cal{E}$:
$$
\Cal E:=H^{1}_0(\Omega)\times L^2(\Omega),\ \ \|\xi_u\|^2_{\Cal E}:=\|\Nx u\|^2_{L^2}+\|\Dt u\|^2_{L^2}.
$$
Dispersive or/and dissipative semilinear wave equations of the form \eqref{0.eqmain}  model various oscillatory processes
in many areas of  modern mathematical physics including  electrodynamics, quantum mechanics, nonlinear elasticity,  etc. and are of a big permanent interest, see \cite{lions,BV,tem,CV,straus,tao,sogge-book} and references therein.
\par
To the best of our knowledge, the global well-posedness of the quintic wave equations in the whole space ($\Omega=\R^3$) has been first obtained by
 Struwe \cite{St}  in the class of radially symmetric solutions and by Grillakis \cite{Grill}   for the non radially symmetric case and smooth initial data. Their proof is strongly based on the explicit formulas for the solutions of the wave equation in $\R^3$ as well as on the so called Morawetz-Pohozhaev identity.
 \par
 The global unique solvability in $\Omega=\R^3$ for the initial data in the energy space has been verified by
 Shatah and Struwe \cite{SS} (see also \cite{SS1} and \cite{kap1,kap2,kap3}). This well-posedness is obtained in the class of solutions which possess (together with the energy estimate) some extra space-time regularity 
(say, $u\in L^4(t,t+1;L^{12}(\Omega))$ or $u\in L^8((t,t+1)\times\Omega)$ or $(u,\Dt u)\in L^4(t,t+1;\dot W^{1/2,4}(\Omega)\times\dot W^{-1/2,4}(\Omega))$) which follow from the Strichartz type estimates. In the present paper we refer to the analogues of such solutions in bounded domains   as {\it Shatah-Struwe} solutions, see Section \ref{s1} for more details.
\par
Again to the best of our knowledge, even in the case of the whole space $\Omega=\R^3$ or in the case where $\Omega$ is a compact manifold without boundary, the global attractors for equations of the type \eqref{0.eqmain} have been constructed only for the {\it sub-quintic} case $f(u)\sim u|u|^{4-\eb}$, $\eb>0$, see \cite{feireisl} and \cite{kap4}, and their existence in the quintic case was a longstanding  open problem.
\par
The case of bounded domains looked even more delicate since the Strichartz type estimates have been not known for that case till recently and based purely on the energy estimates, one can verify the global well-posedness only for the cases of {\it cubic} or {\it sub-cubic} growth rates of the nonlinearity $f$. Therefore, for a long time, exactly the cubic growth rate of the nonlinearity $f$ has been considered as a critical one for the case of bounded domains, see \cite{BV,HR,CV,Lad,ZelDCDS,ZelCPAA, PataZel} and references therein. In particular, the existence of a compact global attractor for that case has been known only for the nonlinearity of the cubic growth rate and, for faster growing nonlinearities, only the versions of weak trajectory attractors (without compactness and uniqueness) have been available, see \cite{CV,ZelDCDS} and references therein.
\par
However, due to the recent progress in Strichartz estimates, see \cite{sogge1,plan1}, the suitable versions of Strichartz estimates are now available for the case of bounded domains as well. Moreover, using also the proper generalization of Morawetz-Pohozhaev identity to the case of bounded domains, the result of Shatah and Struwe on the global well-posedness of quintic wave equations is now extended to the case of smooth bounded domains, see \cite{plan1,plan2}. Thus, it becomes more natural to refer (similar to the case $\Omega=\R^3$)
 to the {\it quintic} growth rate of the non-linearity $f$ as the {\it critical} one and treat the sub-quintic case as a sub-critical one. We will follow this terminology throughout of the paper.
\par
The main aim of the present paper is to develop the attractor theory for the semilinear wave equation \eqref{0.eqmain} in bounded domains for the nonlinearities of the {\it quintic} and sub-quintic (but super-cubic) growth rates. Note from the very beginning that our results in the sub-quintic case are  more or less straightforward extensions of the results \cite{feireisl} to the case of bounded domains based on the new Strichartz estimates. So, we give the analysis of this case only for the completness (see Section \ref{s3}) and are mainly concentrated on the most interesting case of the critical quintic growth rate.
\par
The case of quintic growth rate is indeed much more delicate since the global well-posedness theorem mentioned above gives only the existence and uniqueness of the  solutions with extra space-time regularity $u\in L^4(t,t+1;L^{12}(\Omega))$, but does not give any control of this norm in terms of the initial data and, in particular, does not give any information on the behavior of such norm as $t\to\infty$. By this reason, the control of this norm may be a priori lost when passing to the limit $t\to\infty$. As a result, even starting from the regular Shatah-Struwe solutions, we may a priori lose the extra space-time regularity on the attractor. Since the uniqueness in the classes of solutions weaker than the Shatah-Struwe ones is also not known, this is a crucial difficulty which (again a priori) may destroy the theory.
\par
 To overcome this problem, we verify (in Section \ref{s2}) that any Shatah-Struwe solution can be obtained as a limit of Galerkin approximations and utilize the results obtained in \cite{ZelDCDS} on the weak trajectory attractors of the Galerkin solutions. Namely, based  on the finiteness of the dissipation integral, it is shown there that even in the case of supercritical growth rate of $f$, every complete solution  $u(t)$, $t\in\R$, belonging to the weak attractor becomes {\it smooth} for sufficiently large negative times, see Section \ref{s4}. Combining this result with the global solvability in the class of Shatah-Struwe solutions, we verify that, in the quintic case, the weak attractor consists of smooth solutions which, in particular, satisfy the energy identity. Using then the so-called energy method, see \cite{ball,rosa}, we finally establish the existence of a compact global attractor for the quintic wave equation \eqref{0.eqmain}, see Section \ref{s5}.
 \par
 Thus, the following theorem is the main result of the paper (see Section \ref{s5} for more details).
 \begin{theorem}\label{Th0.main} Let the quintic non-linearity $f$ satisfy assumptions \eqref{2.extradis} and \eqref{00.3} with $p=3$ and let $g\in L^2(\Omega)$. Then, the (Shatah-Struwe) solution semigroup $S(t):\Cal E\to\Cal E$ associated with equation \eqref{0.eqmain} possesses a global attractor $\Cal A$ in $\Cal E$ which is a bounded set in the more regular space
 $$
 \Cal E_1:=[H^2(\Omega)\cap H^1_0(\Omega)]\times H^1_0(\Omega).
 $$
\end{theorem}
The paper is organized as follows. The preliminary things, including the key Strichartz estimates for the linear equation and various types of energy solutions of \eqref{0.eqmain} are discussed in Section \ref{s1}. The key properties of the Shatah-Struwe solutions including the local and global existence, uniqueness and further regularity are collected in Section \ref{s2}. Section \ref{s3} is devoted to the relatively simple sub-critical case when the nonlinearity $f$ grows slower than a quintic polynomial and the analogue of Theorem \ref{Th0.main} for that case is obtained there.
The brief exposition of the trajectory attractor theory for the critical and supercritical wave equations developed in \cite{ZelDCDS} is given in Section \ref{s4}. Finally, the existence of a compact global attractor for the quintic wave equation is proved in Section \ref{s5}.

\section{Preliminaries: Strichartz type estimates and types of energy solutions}\label{s1}
In this section, we introduce the key concepts and technical tools which will be used throughout of the paper. We start with the  estimates for the solutions of the following linear equation:
%$$
\begin{equation}\label{1.linear}
\Dt^2 v+\gamma\Dt v-\Dx v=G(t),\ \ \xi_u\big|_{t=0}=\xi_0,\ \ u\big|_{\partial\Omega}=0.
\end{equation}
%$$
The next proposition is a classical energy estimate for the linear equation \eqref{1.linear}.
\begin{proposition}\label{Prop1.linen} Let $\xi_0\in\Cal E$, $G\in L^1(0,T;L^2(\Omega))$ and let $v(t)$ be a solution of equation \eqref{1.linear} such that $\xi_v\in C(0,T;\Cal E)$. Then the following estimate holds:
%$$
\begin{equation}\label{1.linen}
\|\xi_v(t)\|_{\Cal E}\le C\(\|\xi_0\|_{\Cal E}e^{-\beta t}+\int_0^te^{-\beta(t-s)}\|G(s)\|_{L^2}\,ds\),
\end{equation}
%$$
where the positive constants $C$ and $\beta$ depend on $\gamma>0$, but are independent of $t$, $\xi_0$ and $G$.
\end{proposition}
Indeed, estimate \eqref{1.linen} follows in a standard way by multiplying \eqref{1.linear} by $\Dt v+\alpha v$ (for some $\alpha>0$) and applying the Gronwall inequality, see e.g., \cite{tem,CV} for the details.
\par

The next proposition is however  much more delicate and follows from the recently proved Strichartz type estimates for wave equations in bounded domains, see \cite{sogge1} (see also \cite{plan1,plan2}).

\begin{proposition}\label{Prop1.str} Let the assumptions of Proposition \ref{Prop1.linen} hold. Then, $v\in L^4(0,T;L^{12}(\Omega))$ and the following estimate holds:
%$$
\begin{equation}\label{1.linstr}
\|v\|_{L^4(0,T;L^{12}(\Omega))}\le C_T(\|\xi_0\|_{\Cal E}+\|G\|_{L^1(0,T;L^2(\Omega))}),
\end{equation}
%$$
where $C$ may depend on $T$, but is independent of $\xi_0$ and $G$.
\end{proposition}
Indeed, for $\gamma=0$ this estimate is established in \cite{sogge1} and the case $\gamma\ne0$ is reduced to the case $\gamma=0$ due to the control of the $L^2$-norm of $\Dt v$ via energy estimate \eqref{1.linen}.

\begin{remark}\label{Rem1.lindis} Combining energy estimate \eqref{1.linen} with the Strichartz estimate \eqref{1.linstr}, we get  a bit stronger {\it dissipative} version of \eqref{1.linstr}:
%$$
\begin{equation}\label{1.disstr}
\|\xi_v(t)\|_{\Cal E}+\|v\|_{L^4(\max\{0,t-1\},t;L^{12}(\Omega))}\le C\(\|\xi_0\|_{\Cal E}e^{-\beta t}+\int_0^te^{-\beta(t-s)}\|G(s)\|_{L^2}\,ds\),
\end{equation}
%$$
where positive constant $\beta$ and $C$ are independent of $v$ and $t\ge0$.
\par
Note also that, due to the interpolation inequality
%$$
\begin{equation}\label{1.int}
\|v\|_{L^{\frac4\theta}(0,T;L^{\frac{12}{2-\theta}}(\Omega))}\le C\|v\|_{L^4(0,T;L^{12}(\Omega))}^\theta\|v\|_{L^\infty(0,T;H^1(\Omega))}^{1-\theta}
\end{equation}
%$$
and energy estimate \eqref{1.linen}, we have the control of the $L^{4/\theta}(L^{12/(2-\theta)})$-norm of the solution $v$ for all $\theta\in[0,1]$. Most important for what follows will be the case $\theta=\frac45$ which controls the $L^5(L^{10})$-norm of the solution.
\end{remark}
The next elementary fact will be  used below for verifying the local existence of weak solutions.
\begin{corollary}\label{Cor1.comp} Let $K\subset\Cal E\times L^1(0,T;L^2(\Omega))$ be a compact set. Then, for every $\eb>0$ there is $T(\eb)>0$ such that
%$$
\begin{equation}\label{1.strcomp}
\|v\|_{L^4(0,T(\eb);L^{12}(\Omega))}\le \eb
\end{equation}
%$$
for all solutions $v$ of problem \eqref{1.linear} with $(\xi_0,G)\in K$.
\end{corollary}
Indeed, this assertion is an immediate corollary of estimate \eqref{1.linstr} and the Hausdorff criterium.
\begin{remark} It is not difficult to show, using e.g. the scaling arguments that the assertion of Corollary \ref{Cor1.comp} is false in general if the set $K$ is only bounded in $\Cal E\times L^1(0,T;L^2(\Omega))$.
\end{remark}
We now turn to the nonlinear problem \eqref{0.eqmain} with the nonlinearity of quintic growth rate:
%$$
\begin{equation}\label{1.f5}
|f'(u)|\le C(1+|u|^4)
\end{equation}
%$$
and discuss several classes of weak solutions for it. The most straightforward definition is the following one.
\begin{definition}\label{Def1.simple} A function $v(t)$ is a weak (energy) solution of problem \eqref{0.eqmain} if $\xi_v\in L^\infty(0,T;\Cal E)$ and equation \eqref{0.eqmain} is satisfied in the sense of distributions. The latter means that
%$$
\begin{equation}\label{1.distr}
-\int_0^T(u_t,\phi_t)dt-\gamma\int_0^T(u,\phi_t)dt+\int_0^T(\nabla u,\nabla \phi)dt+\int_0^T(f(u),\phi)dt=\int_0^T(g,\phi)\,dt,
\end{equation}
%$$
for any $\phi\in C_0^\infty((0,T)\times\Omega)$. Here and below $(u,v)$ stands for the usual inner product in $L^2(\Omega)$. Then, due to the growth restriction \eqref{1.f5},
$$
f(u)\in L^\infty(0,T;H^{-1}(\Omega))
$$
and from \eqref{0.eqmain}, we conclude that $\Dt^2 u\in L^\infty(0,T;H^{-1}(\Omega))$. Thus,
$$
\xi_u(t)\in C(0,T;\Cal E_{-1}),\ \ \ \Cal E_{-1}:=L^2(\Omega)\times H^{-1}(\Omega)
$$
and the initial condition $\xi_u(0)=(u_0,u_0')$ is well-defined.
\end{definition}
However, these solutions are extremely difficult to work with. Indeed, most part of estimates related with equation \eqref{0.eqmain} are based on energy type estimates and this requires multiplication of \eqref{0.eqmain} by $\Dt u$, but the regularity of energy solutions is not enough to justify this multiplication if $f(u)$ has faster than cubic growth rate. Thus, to the best of our knowledge even the basic energy estimate is not known for such solutions if $f$ grows faster than $u^3$.
\par
At least two alternative ways to overcome this problem has been used in a literature. One of them consists of requiring additionally a weak solution to satisfy most important energy equalities or inequalities (see \cite{CV,MirZel,plan3} and reference therein). The other one poses the extra condition that a weak solution is obtained as a limit of {\it smooth} solutions of the properly chosen approximation problems. Then the desired estimates are obtained by passing to the limit from the analogous estimates for the approximating solutions, see e.g., \cite{ZelDCDS,MirZel} and reference therein. In this paper we will use the so-called Galerkin approximations for that purposes.
\par
Let $\lambda_1\le\lambda_2\le\cdots$ be the eigenvalues of the operator $-\Dx$ with homogeneous  Dirichlet boundary conditions and $e_1,e_2,\cdots$ be the corresponding eigenfunctions. Then, they form an orthonormal base in $L^2(\Omega)$ and since the domain $\Omega$ is smooth, they are also smooth: $e_i\in C^\infty(\Omega)$. Let $P_N:L^2(\Omega)\to L^2(\Omega)$ be the orthoprojector to the linear subspace spanned by the first $N$ eigenfunctions $\{e_1,\cdots,e_N\}$. Then, the Galerkin approximations to problem \eqref{0.eqmain} are defined as follows:
%$$
\begin{equation}\label{1.galerkin}
\begin{cases}
\Dt^2u_N+\gamma\Dt u_N-\Dx u_N+P_Nf(u_N)=P_Ng,\ \ u_N\in P_NL^2(\Omega),\\
\xi_{u_N}(0)=\xi_0^N\in [P_NL^2(\Omega)]^2.
\end{cases}
\end{equation}
%$$
Remind that \eqref{1.galerkin} is a system of ODEs of order $2N$ with smooth (at least $C^1$) nonlinearity, so it is {\it locally} uniquely solvable and under some natural dissipativity assumptions on $f$ (e.g., \eqref{2.extradis}) the blow up is impossible and the solution is globally defined as well. Moreover, since all of the eigenvectors $e_i$ are smooth, the solutions $u_N(t,x)$ are $C^\infty$-smooth in $x$ and give indeed the desired smooth approximations to \eqref{0.eqmain}. This justifies the following definition

\begin{definition}\label{Def1.Gal} A weak solution $u(t)$, $t\in[0,T]$ (in the sense of Definition \ref{Def1.simple}) is called {\it Galerkin} (weak) solution of problem \eqref{0.eqmain} if it can be obtained as a weak-star limit in $L^\infty(0,T;\Cal E)$ of the Galerkin approximation solutions  $u_N$ of problems \eqref{1.galerkin}:
%$$
\begin{equation}\label{1.gallim}
\xi_u=\lim_{N\to\infty}\xi_{u_n},
\end{equation}
%$$
where the limit is taken in the weak-star topology of $L^\infty(0,T;\Cal E)$. Note that this convergence implies only that
%$$
\begin{equation}\label{1.weak}
\xi_{u_N}(0)\rightharpoonup\xi_u(0)
\end{equation}
%$$
in $\Cal E$ and the strong convergence of the initial data in the energy space is not assumed, see \cite{ZelDCDS} for more details.
\end{definition}
Although the Galerkin solutions are a priori more friendly than the general weak solutions, their uniqueness is known only for the non-linearities growing not faster than $u^3$, so for faster growing nonlinearities, one should use the so-called {\it trajectory} attractors for study their long-time behavior, see \cite{ZelDCDS,MirZel} and also Section \ref{s4} below for more details.
\par
The uniqueness and global well-posedness problem for the case $\Omega=\R^3$ has been resolved by Shatah and Struwe \cite{SS} (see also \cite{kap1,kap2}) in the class of weak solutions satisfying additionally some space-time regularity estimate (e.g., $u\in L^4(0,T;L^{12}(\Omega))$). This result is strongly based on Strichartz estimates for the linear wave equations as well as the Morawetz identity for the nonlinear equation. The analogues of that results for the case of bounded domains have been recently obtained in \cite{sogge1,plan1}, so analogously to the case $\Omega=\R^3$, one can give the following definition (see \cite{sogge1,plan1}).

\begin{definition}\label{Def1.SS} A weak solution $u(t)$, $t\in[0,T]$ is a {\it Shatah-Struwe} solution of problem \eqref{0.eqmain} if the following additional regularity holds:
%$$
\begin{equation}\label{1.str-reg}
u\in L^4(0,T;L^{12}(\Omega)).
\end{equation}
%$$
\end{definition}
\begin{remark}\label{Rem1.equiv} As we will see below, the introduced Shatah-Struwe solutions is a natural class of solutions where the global well-posedness, dissipativity and asymptotic smoothness can be established. However, to verify the existence of a global attractor we will essentially use the Galerkin solutions as an intermediate technical tool.
\par
Note also that the ideal situation where all three introduced above classes of weak solutions are in fact {\it equivalent} is not a priori excluded although, to the best of our knowledge, that is rigorously proved only for the nonlinearities growing not faster than $u^3$. Some results in this direction for quintic nonlinearities and $\Omega=\R^3$ are  obtained in \cite{plan3}.
\end{remark}
\section{Properties of Shatah-Struwe solutions}\label{s2}
The aim of this section is to discuss the well-posednes, dissipativity and smoothness of Shatah-Struwe solutions of problem \eqref{0.eqmain}. Although most of these results are not new or follow in a straightforward way from the known results, they are crucial for what follows, so for the convenience of the reader, we give their proofs here.
\par
We start with the local existence result.

\begin{proposition}\label{Prop2.exist} Let $g\in L^2$ and the nonlinearity $f$ satisfy the growth assumption \eqref{1.f5}. Then, for any initial data $\xi_0\in\Cal E$, there exists $T=T(\xi_0)>0$ such that problem \eqref{0.eqmain} possesses a Shatah-Struwe solution $u(t)$ on the interval $t\in[0,T]$.
\end{proposition}
\begin{proof} We construct the desired solution $u$ by passing to the limit $N\to\infty$ in the Galerkin approximations \eqref{1.galerkin}. To this end, it suffices  to obtain a {\it uniform} with respect to $N$ estimate for the $L^4(0,T;L^{12})$-norm of the solutions $u_N(t)$. To obtain such an estimate, we fix the initial data $\xi_{u_N}(0)=\xi_0^N:=P_N\xi_0$. This guarantees that
$$
\xi_0^N\to\xi_0
$$
{\it strongly} in $\Cal E$. Then, we split the solution $u_N=v_N+w_N$ where $v_N$ solves the linear problem
%$$
\begin{equation}
\Dt^2 v_N+\gamma\Dt v_N-\Dx v_N=P_Ng,\ \ \xi_{u_N}(0)=P_N\xi_0
\end{equation}
%$$
and $w_N$ is a reminder which satisfies
%$$
\begin{equation}\label{2.wN}
\Dt^2w_N+\gamma\Dt w_N-\Dx w_n=-P_N f(v_N+w_N),\ \ \xi_{w_N}(0)=0.
\end{equation}
%$$
Note that the set of data $\{(\xi_0^N,P_Ng)\}_{N=1}^\infty$ is a compact set in $\Cal E\times L^1(0,1;L^2(\Omega))$. Therefore, due to Corollary \ref{Cor1.comp}, for any $\eb>0$, there exists $T=T(\eb)>0$ (which is independent of $\eb$) such that
%$$
\begin{equation}\label{2.Nstr}
\|\xi_{v_N}(t)\|_{\Cal E}\le C,\ \ \|v_N\|_{L^4(0,t;L^{12}(\Omega))}\le \eb,\ \ \ t\le T(\eb),
\end{equation}
%$$
where $C$ is independent of $N$. Then, due to the growth restriction \eqref{1.f5} and interpolation inequality \eqref{1.int} with $\theta=4/5$, we have
%$$
\begin{multline}\label{3.pnf}
\|P_Nf(v_N+w_N)\|_{L^1(0,t;L^{2}(\Omega))}\le\\\le C(t+\|v_N\|_{L^5(0,t;L^{10}(\Omega))}^5+\|w_N\|_{L^5(0,t;L^{10}(\Omega))}^5)\le C(t+\eb^4+\|w_N\|_{L^5(0,t;L^{10}(\Omega))}^5).
\end{multline}
%$$
Applying now estimate \eqref{1.disstr} to equation \eqref{2.wN} and using interpolation inequality \eqref{1.int} with $\theta=4/5$ again, we end up with
%$$
\begin{multline}\label{2.wNgood}
\|\xi_{w_N}(t)\|_{\Cal E}+\|w_N\|_{L^5(0,t;L^{10}(\Omega))}+\|w_N\|_{L^4(0,t;L^{12}(\Omega))}\le\\\le C(t+\eb^4)+C(\|\xi_{w_N}(t)\|_{\Cal E}+\|w_N\|_{L^5(0,t;L^{10}(\Omega))})^5.
\end{multline}
%$$
Thus, denoting $Y_N(t):=\|\xi_{w_N}(t)\|_{\Cal E}+\|w_N\|_{L^5(0,t;L^{10}(\Omega))}$, we end up with the inequality
$$
Y_N(t)\le C(t+\eb^4)+CY_N(t)^5,\ \ t\le T(\eb),\ \ Y_N(0)=0,
$$
where the constant $C$ is independent of $N$, $t$ and $\eb$. Moreover, obviously $Y_N(t)$ is a continuous function of $t$. Then, the last inequality gives
$$
Y_N(t)\le 2C(t+\eb^4)
$$
if $\eb$ and $T(\eb)$ is chosen in such way that
$$
C(2C(t+\eb^4))^5\le C(t+\eb^4),\ \ t\le T(\eb).
$$
Thus, fixing $\eb$ and $T=T(\eb)$ being small enough to satisfy the last inequality, we get the uniform estimate
$$
\|\xi_{w_N}\|_{C(0,T;\Cal E)}+\|w_N\|_{L^5(0,T;L^{10}(\Omega))}\le C_1
$$
which together with  \eqref{2.Nstr} and \eqref{2.wNgood}, gives the desired uniform estimate
$$
\|\xi_{u_N}\|_{L^\infty(0,T;\Cal E)}+\|u_N\|_{L^4(0,T;L^{12}(\Omega))}\le C_2.
$$
Passing then in a standard way to the weak limit $N\to\infty$, we end up with a Shatah-Struwe solution of \eqref{0.eqmain} and finish the proof of the proposition.
\end{proof}
\begin{remark}\label{Rem2.difficult} Obviously, the lifespan $T=T(\xi_0)$ of a Shatah-Struwe solution $u$ depends a priori on the initial data $\xi_0$ and since it is greater than zero for all $\xi_0\in\Cal E$, one may expect that $T$ depends on the $\Cal E$-norm of $\xi_0$ only:
%$$
\begin{equation}\label{2.false}
T=T(\|\xi_0\|_{\Cal E}).
\end{equation}
%$$
In that case, the {\it global} solvability problem would be reduced to the control of the energy norm of a solution $u(t)$. Since such a control follows immediately from the energy estimate (see below), the global solvability would also become immediate.
\par
Unfortunately, \eqref{2.false} {\it is not true} in the critical quintic case (as we will see in the next section, it is indeed true in the subcritical case) and, by this reason, one needs a lot of extra efforts in order to establish the desired global solvability.
\par
However, \eqref{2.false} remains true even in the critical quintic case under the extra assumption that the energy norm $\|\xi_0\|_{\Cal E}$ is small:
$$
\|\xi_0\|_{\Cal E}\le\eb_1\ll1.
$$
Indeed, in that case the key estimates \eqref{2.Nstr}, follow directly from the smallness of the energy norm of $\xi_0$ and estimate \eqref{1.disstr} and no compactness arguments of Corollary \ref{Cor1.comp} are required. This simple observation not only leads for the global solvability for small initial data in the case where the smallness of the energy of a solution follows from the energy estimate, but also plays a crucial role in the proof of global solvability for all initial data via the non-concentration of the energy norm, see \cite{plan1,SS}.
\end{remark}
At the next step, we check that a Shatah-Struwe solution satisfies the energy equality.
\begin{proposition}\label{Prop2.eneq} Let $g\in L^2(\Omega)$, the nonlinearity $f$ satisfy the growth restriction \eqref{1.f5} and $u(t)$, $t\in[0,T]$ be a Shatah-Struwe solution of equation \eqref{0.eqmain}. Then the functions $t\to\|\xi_u(t)\|_{\Cal E}$ and $t\to (F(u(t)),1)$ are
absolutely continuous and the following energy identity
%$$
\begin{equation}\label{2.energy}
\frac d{dt}\(\frac12\|\xi_u(t)\|^2_{\Cal E}+(F(u(t)),1)-(g,u(t))\)+\gamma\|\Dt u(t)\|^2_{L^2}=0
\end{equation}
%$$
holds for almost all $t\in[0,T]$. In particular, $\xi_u\in C([0,T],\Cal E)$.
\end{proposition}
\begin{proof} Indeed, due to the definition of a Shatah-Struwe solution, growth restriction \eqref{1.f5} and interpolation inequality \eqref{1.int}, we have
%$$
\begin{equation}\label{2.f12}
\|f(u)\|_{L^1(0,T;L^2(\Omega))}\le C(1+\|u\|_{L^5(0,T;L^{10}(\Omega))})
\le C_T.
\end{equation}
%$$
Therefore, since $\Dt u\in L^\infty(0,T;L^2(\Omega))$, $f(u)\Dt u\in L^1([0,T]\times\Omega)$, then approximating the function $u$ by smooth ones and arguing in a standard way, we  see that for every $0\le\tau\le t\le T$,
$$
(F(u(t)),1)-(F(u(\tau)),1)=\int_\tau^t(f(u(s)),\Dt u(s))\,ds
$$
and, consequently, $t\to (F(u(t)),1)$ is absolutely continuous and
%$$
\begin{equation}\label{2.fabs}
\frac d{dt}(F(u(t)),1)=(f(u(t)),\Dt u(t))
\end{equation}
%$$
for almost all $t\ge0$.
\par
We are now ready to finish the proof of energy equality. To this end, we  take $u_N(t):=P_Nu(t)$, where $P_N$ is the orthoprojector on the first $N$ eigenvalues of the Laplacian. Then, this function solves
$$
\Dt^2 u_N+\alpha\Dt u_N+\gamma\Dt u_N-\Dx u_N=-P_N f(u)+P_Ng.
$$
Multiplying this equation by $\Dt u_N$ and integrating in space and time, we get the following analogue of energy equality:
%$$
\begin{multline}\label{2.nenergy}
\frac12\|\xi_{u_N}(t)\|^2_{\Cal E}-\frac12\|\xi_{u_N}(\tau)\|^2_{\Cal E}-(g,u_N(t))+(g,u_N(\tau))+\\+\int_\tau^t\gamma\|\Dt u_N(s)\|^2_{L^2}\,ds=-\int_\tau^t(P_Nf(u(s)),\Dt u_N(s))\,ds.
\end{multline}
%$$
Since, obviously, $\xi_{u_N}(t)\to\xi_u(t)$, $\xi_{u_N}(\tau)\to\xi_u(\tau)$ and $\Dt u_N\to\Dt u$ in $L^2(\tau,t;L^2(\Omega))$ strongly and
$P_Nf(u)\to f(u)$ weakly-star in $L^1(\tau,t;L^2(\Omega))$, we may pass to the limit $N\to\infty$ in \eqref{2.nenergy} and with the help of \eqref{2.fabs}  obtain that
%$$
\begin{equation}\label{2.eqv}
E(u(t))-E(u(\tau))+\gamma\int_\tau^t\|\Dt u(s)\|^2_{L^2}\,ds=0,
\end{equation}
%$$
where
$$
E(u):=\frac12\|\xi_u\|^2_{\Cal E}+(F(u),1)-(g,u).
$$
It remains to note that \eqref{2.eqv} is equivalent to \eqref{2.energy} and the energy equality is proved. The continuity of $\xi_u(t)$ as a $\Cal E$-valued function follows in a standard way from the energy equality.
\end{proof}
\begin{corollary}\label{Cor2.uni} Let the assumptions of Proposition \ref{Prop2.eneq} hold. Then the Shatah-Struwe solution $u(t)$, $t\in[0,T]$ is unique.
\end{corollary}
\begin{proof} Indeed, let $u(t)$ and $v(t)$ be two Shatah-Struwe solutions of equation \eqref{0.eqmain} on the interval $t\in[0,T]$ and let $w(t)=u(t)-v(t)$. Then this function solves
%$$
\begin{equation}\label{2.dif}
\Dt^2 w+\gamma\Dt w-\Dx w+[f(u)-f(v)]=0.
\end{equation}
%$$
Multiplying this equation by $\Dt w$ and integrating over $x\in\Omega$ (which is justified exactly as in Proposition \ref{Prop2.eneq}), we end up with
%$$
\begin{equation}\label{2.difeq}
\frac12\frac d{dt}\|\xi_w(t)\|^2_{\Cal E}+\gamma\|\Dt w(t)\|^2_{L^2}+(f(u)-f(v),\Dt w)=0.
\end{equation}
%$$
Using the growth restriction \eqref{1.f5}, the Sobolev embedding $H^1_0\subset L^6$ and the H\"older inequality, we estimate the last term at the left-hand side of \eqref{2.difeq} as follows:
%$$
\begin{multline}\label{2.SSdif}
|(f(u)-f(v),\Dt w)|\le C((1+|u|^4+|v|^4)|w|,|\Dt w|)\le\\\le C(1+\|u\|_{L^{12}}^4+\|v\|^4_{L^{12}})\|w\|_{L^6}\|\Dt w\|_{L^2}\le C(1+\|u\|_{L^{12}}^4+\|v\|^4_{L^{12}})\|\xi_w\|^2_{\Cal E}.
\end{multline}
%$$
Since the $L^4(0,T;L^{12}(\Omega))$-norms of $u$ and $v$ are finite by the definition of the Shatah-Struwe solutions, inserting the obtained estimate into equality \eqref{2.difeq} and applying the Gronwall inequality, we see that $\xi_w(t)\equiv0$ and the corollary is proved.
\end{proof}
\begin{corollary}\label{Cor2.dis} Let the assumptions of Proposition \ref{Prop2.eneq} hold and let, in addition, the nonlinearity $f$ satisfies the following dissipativity assumption:
%$$
\begin{equation}\label{2.fdis}
f(u)u\ge -C,\ \ u\in\R.
\end{equation}
%$$
Then the Shatah-Struwe solution $u(t)$, $t\in[0,T]$ of problem \eqref{0.eqmain} satisfies the following dissipative estimate:
%$$
\begin{equation}\label{2.disest}
\|\xi_u(t)\|_{\Cal E}\le Q(\|\xi_u(0)\|_{\Cal E})e^{-\alpha t}+Q(\|g\|_{L^2}),\ \ t\in[0,T],
\end{equation}
%$$
where the monotone function $Q$ and positive constant $\alpha$ are independent of $t$, $T$ and $u$.
\end{corollary}
Indeed, the energy estimate \eqref{2.disest} follows in a standard way by multiplication of equation \eqref{0.eqmain} by $\Dt u+\beta u$, where $\beta$ is a properly chosen positive constant, and integration in $x$ (the validity of that is verified in Proposition \ref{Prop2.eneq}) followed by application of the Gronwall type inequality, see \cite{BV,CV,ZelDCDS} for more details.
\par
The next corollary shows that Shatah-Struwe solutions can be obtained as a limit of Galerkin approximations.

\begin{corollary}\label{Cor2.SSGalerkin} Let the assumptions of Proposition \ref{Prop2.eneq} hold, the nonlinearity $f$ satisfy the dissipativity assumption \eqref{2.fdis} and let $u(t)$ be a Shatah-Struwe solution of problem \eqref{0.eqmain}. Assume also that the initial data $\xi_{u_N}(0)\in P_N\Cal E$ for the Galerkin approximations $u_N(t)$ are chosen in such way that
$$
\xi_{u_N}(0)\to\xi_u(0)
$$
strongly in $\Cal E$. Then, the Galerkin solutions $u_N(t)$ converge to the solution $u(t)$:
%$$
\begin{equation}\label{2.SSstrong}
\xi_{u_N}(t)\to\xi_u(t)
\end{equation}
%$$
strongly in $\Cal E$ for every $t\in[0,T]$. In particular, any Shatah-Struwe solution is a Galerkin solution of problem \eqref{0.eqmain}.
\end{corollary}
\begin{proof} Indeed, due to the energy estimate \eqref{2.disest} for the Galerkin approximations $u_N(t)$, we know that the $L^\infty(0,T;\Cal E)$-norms of these solutions are uniformly bounded. Thus, we may assume without loss of generality, that $u_N(t)\to\bar u(t)$ weakly star in
$L^\infty(0,T;\Cal E)$, where $\bar u(t)$ is a weak energy solution of equation \eqref{0.eqmain}. Moreover, arguing in a standard way, we see that
%$$
\begin{equation}\label{2.weak}
\xi_{u_N}(t)\rightharpoonup\xi_{\bar u}(t)
\end{equation}
%$$
in $\Cal{E}$ for every $t\in[0,T]$. In addition, from the proof of Proposition \ref{Prop2.exist}, we know that $\bar u(t)$ is  a Shatah-Struwe solution for $t\in[0,T_1]$ for some small, but positive $T_1$. Then, by Corollary \ref{Cor2.uni}, $u(t)=\bar u(t)$, $t\in[0,T_1]$. Introduce the time
$$
T^*:=\sup\{t\in[0,T],\ u(s)=\bar u(s),\ s\le t\}.
$$
We need to prove that $T^*=T$. To this end, we note that
$$
\xi_{u_N}(T^*)\to\xi_u(T^*).
$$
The weak convergence follows from \eqref{2.weak} and to verify the strong convergence, we check that
%$$
\begin{equation}\label{2.strong}
\|\xi_{u_N}(T^*)\|_{\Cal E}\to\|\xi_{u}(T^*)\|_{\Cal E}.
\end{equation}
%$$
Assume that \eqref{2.strong} is wrong, then without loss of generality, we may assume that
%$$
\begin{equation}\label{2.wrong}
\|\xi_{u_N}(T^*)-\xi_u(T^*)\|_{\Cal E}\ge\eb_0>0.
\end{equation}
%$$
Then, we want to pass to the limit in the energy equality
%$$
\begin{multline}\label{2.uN}
\frac12\|\xi_{u_N}(T^*)\|^2_{\Cal E}+(F(u_N(T^*)),1)-(g,u_N(T^*))+\gamma\int_0^{T^*}\|\Dt u_N(t)\|^2_{L^2}\,dt=\\=\frac12\|\xi_{u_N}(0)\|^2_{\Cal E}+(F(u_N(0)),1)-(g,u_N(0))
\end{multline}
%$$
for Galerkin approximations $u_N$.
Indeed, since we have the {\it strong} convergence $\xi_{u_N}(0)\to\xi_u(0)$, the right-hand side of \eqref{2.uN} tends to the analogous expression for $u$. To pass to the limit in the left hand side, we use the inequality
%$$
\begin{equation}\label{2.weakxi}
\|\xi_{u}(T^*)\|_{\Cal E}\le \liminf_{N\to\infty}\|\xi_{u_N}(T^*)\|_{\Cal E}
\end{equation}
%$$
which is valid due to the weak convergence \eqref{2.weak}, and
%$$
\begin{equation}\label{2.weakF}
(F(u(T^*),1)\le \liminf_{N\to\infty}(F(u_N(T^*),1),
\end{equation}
%$$
due to the fact that $u_N(T^*)\to u(T^*)$ almost everywhere, assumption \eqref{2.fdis} and the Fatou lemma. Thus,
%$$
\begin{multline}\label{2.uNeq}
\frac12\|\xi_{u}(T^*)\|^2_{\Cal E}+(F(u(T^*),1)-(g,u(T^*))+\gamma\int_0^{T^*}\|\Dt u(t)\|^2_{L^2}\,dt\le\\\le\frac12\|\xi_{u}(0)\|^2_{\Cal E}+(F(u(0),1)-(g,u(0)).
\end{multline}
%$$
On the other hand, $u(t)$ is a Shatah-Struwe solution and, by this reason it satisfies the energy equality. Thus, the inequality in \eqref{2.uNeq} is actually the equality which is possible only when both \eqref{2.weakxi} and \eqref{2.weakF} are also equalities.
\par
 In particular, for some subsequence $N_k$, we have $\|\xi_{u_{N_k}}(T^*)\|_{\Cal E}\to\|\xi_{u}(T^*)\|_{\Cal E}$ and together with \eqref{2.weak}, we have the strong convergence $\xi_{u_{N_k}}(T^*)\to\xi_u(T^*)$ which contradicts \eqref{2.wrong}. Thus, the {\it strong} convergence $\xi_{u_N}(T^*)\to\xi_{u}(T^*)$ is proved.
\par
Finally, using this strong convergence and arguing as in Proposition \ref{Prop2.exist}, we see that $\bar u(t)$ is a Shatah-Struwe solution on the interval $[T^*,T^*+T_2]$, for some positive $T_2$ and, therefore, should coincide with $u(t)$ on that interval as well. This contradiction shows that, actually, $T^*=T$ and $u(t)=\bar u(t)$ for all $t\in[0,T]$. The convergence \eqref{2.SSstrong} can be then proved based on the energy equality exactly as it was done before for the case $t=T^*$. Corollary \ref{Cor2.SSGalerkin} is proved.
\end{proof}
\begin{remark}\label{Rem2.uniform} Arguing in a bit more accurate way, one can show that, under assumptions of the previous corollary, $\xi_{u_N}\to\xi_u$ strongly in $C(0,T;\Cal E)$ as well.
\end{remark}
We are  ready to state the main result of the section on the global existence of Shatah-Struwe solutions.
\begin{theorem}\label{Th2.SSgwp} Let $g\in L^2(\Omega)$ and the nonlinearity $f$ satisfy
assumptions \eqref{1.f5} and \eqref{2.fdis}. Assume also that the following extra dissipativity assumptions are satisfied
%$$
\begin{equation}\label{2.extradis}
\begin{cases}
1. \ \ F(u)\ge-C+\kappa|u|^6,\ \ \kappa>0,\\
2. \ \ f(u)u-4F(u)\ge-C.
\end{cases}
\end{equation}
%$$
Then, for any $\xi_0\in\Cal E$ there exists a unique Shatah-Struwe solution $u(t)$ defined for all $t\in\R_+$ and this solution satisfies \eqref{2.disest} as well as the following Strichartz type estimate:
%$$
\begin{equation}\label{2.badSS}
\|u\|_{L^4(0,T;L^{12}(\Omega))}\le Q(\xi_0,T),\ \ T\ge0
\end{equation}
%$$
for some function $Q$ monotone increasing in $T$.
\end{theorem}
Indeed, it only remains to prove the global solvability for \eqref{0.eqmain} in the class of Shatah-Struwe solutions.
The proof of this fact was given in \cite{plan1} for the particular case $f(u)=u^5$ and $\gamma=0$, $g=0$ and is based on proving the energy non-concentration for $u(t)$ via the Morawetz type identities adapted to the case of bounded domains. The general case can be treated by repeating verbatim   the arguments of \cite{plan1} and is left  to the reader. Note also that the extra dissipativity assumptions \eqref{2.extradis} do not allow the function $f(u)$ to grow slower than $u^5$, however, it is actually not a big restriction since the most difficult is exactly the case of critical quintic growth rate and as we will see in the next section, we do not  need  Theorem \ref{Th2.SSgwp} to treat the subcritical case.
\par
We are now ready to define the solution semigroup $S(t):\Cal E\to\Cal E$ associated with equation \eqref{0.eqmain}:
%$$
\begin{equation}\label{2.semigroup}
S(t)\xi_0:=\xi_u(t),
\end{equation}
%$$
where $u(t)$ is a unique Shatah-Struwe solution of \eqref{0.eqmain}. Then, according to Theorem \ref{Th2.SSgwp}, this semigroup is well-defined
(and even locally Lipschitz continuous in $\Cal E$, see Corollary \ref{Cor2.uni}) and is dissipative:
%$$
\begin{equation}\label{2.disSS}
\|S(t)\xi_0\|_{\Cal E}\le Q(\|\xi_0\|_{\Cal E})e^{-\alpha t}+Q(\|g\|_{L^2}).
\end{equation}
%$$
\begin{remark}\label{Rem2.problems} The long time behavior of the solution semigroup $S(t)$ will be studied in the next sections. However, it worth to note here that we have the dissipative estimate for the {\it energy} norm  only and Theorem \ref{Th2.SSgwp} gives us
no control of the Strichartz norm as $T\to\infty$ (it's proof is a typical proof ad absurdum which gives no bounds on the function $Q$ in estimate \eqref{2.badSS}). By this reason, the Strichartz estimate may a priori disappear when passing to the limit $t\to\infty$ and the attractor may consist not only of Shatah-Struwe solutions. The proof that it is actually not the case is one the main tasks of the present paper.
\end{remark}
We conclude this section by one more result which shows that a Shatah-Struwe solution is more regular if the initial data is smoother.

\begin{proposition}\label{Prop2.smooth} Let the assumptions of Theorem \ref{Th2.SSgwp} hold and let, in addition, the non-linearity satisfy the following condition:
%$$
\begin{equation}\label{2.technical}
f'(u)\ge-K
\end{equation}
%$$
and the initial data be more smooth, i.e.,
%$$
\begin{equation}\label{2.inism}
\xi_0\in\Cal E_1:=[H^2(\Omega)\cap H^1_0(\Omega)]\times H^1_0(\Omega).
\end{equation}
%$$
Then, the corresponding Shatah-Struwe solution is more regular as well:
%$$
\begin{equation}\label{2.smSS}
 \xi_u(t)=(u(t),u_t(t))\in\Cal E_1
\end{equation}
%$$
for all $t\ge0$.
\end{proposition}
\begin{proof} We give below only the formal proof which can be justified using Galerkin approximations and Corollary \ref{Cor2.SSGalerkin}.
Indeed, let $v(t):=\Dt u(t)$. Then, as not difficult to check using equation \eqref{0.eqmain} and the growth restriction \eqref{1.f5},
 $$
 \xi_v(0)=(\Dt u(0),\Dt^2 u(0)=(u_0',\Dx u_0-f(u_0)-\gamma u_0'+g)\in\Cal E
 $$
 and the function $v$ solves
 %$$
\begin{equation}\label{2.diffeq}
\Dt^2 v+\gamma\Dt v-\Dx v=-f'(u)v,\ \ \xi_v(0)\in\Cal E.
\end{equation}
%$$
Multiplying equation \eqref{2.diffeq} by $\Dt v$, we get
%$$
\begin{equation}\label{2.02}
\frac{1}{2}\frac{d}{dt}\|\xi_v(t)\|^2_{\Cal E}+\gamma\|\Dt v\|^2_{L^2}=-(f'(u)v,\Dt v).
\end{equation}
%$$
Due to the growth restriction \eqref{1.f5}, the term on the right hand side of \eqref{2.02} obeys the estimate, see \eqref{2.SSdif},
%$$
\begin{equation}\label{2.03}
|(f'(u)v,\Dt v)|\le C((1+|u|^4)|v|,|\Dt v|)\le
C(1+\|u\|^4_{L^{12}})\|\xi_v\|^2_{\Cal E}.
\end{equation}
%$$
Substituting the above estimate to \eqref{2.02} and using Gronwall inequality one gets
%$$
\begin{equation}
\|\xi_v(t)\|^2_\Cal{E}\leq \|\xi_v(0)\|^2_\Cal{E} \operatorname{exp}\(CT+\int_0^T\|u(s)\|^4_{L^{12}}ds\),\quad 0\leq t\leq T.
\end{equation}
%$$
The fact that $\xi_v(t)\in\Cal E$, in turn, implies that $\xi_u(t)\in\Cal E_1$. Indeed, the fact that $v=\Dt u\in H^1$ is immediate and we only need to check that $u\in H^2$. To this end, we rewrite equation \eqref{0.eqmain} in the form
%$$
\begin{equation}\label{2.ell}
\Dx u(t)-f(u(t))=g-\Dt v(t)-\gamma v(t):=g_v(t)\in L^2(\Omega)
\end{equation}
%$$
and, multiplying this elliptic equation by $\Dx u$ in $L^2(\Omega)$ and using the additional assumption \eqref{2.technical}, we end up with
%$$
\begin{equation}\label{2.maximal}
\|u(t)\|_{H^2}^2\le C\|\Dx u(t)\|_{L^2}^2\le C\|g\|^2_{L^2}+\|\xi_v(t)\|^2_{\Cal E}+K\|\xi_u(t)\|^2_{\Cal E}.
\end{equation}
%$$
Thus, Proposition \ref{Prop2.smooth} is proved.
\end{proof}

\begin{remark}\label{Rem2.technical} The extra assumption \eqref{2.technical} on the non-linearity $f$ is not essential and is introduced only in order to avoid the technicalities related with the maximal regularity estimate for the critical elliptic equation \eqref{2.ell}. Indeed, under this extra assumption, it is immediate as we have seen. In the general case when $f$ satisfies only the growth restriction \eqref{1.f5} it is also true, but its proof is much more delicate and requires, to use  e.g., the localization in space technique which we did not want to discuss here.
\par
Note also that after obtaining the $H^2$-estimate for $u(t)$, the growth rate of $f$ becomes not essential due to the embedding $H^2\subset C$, so the further regularity of the solution $u(t)$ (if the initial data is more smooth) can be obtained by usual bootstrapping arguments.
\end{remark}

\section{Asymptotic compactness and attractors: the subcritical case}\label{s3}
The aim of this section is to consider the subcritical case where the nonlinearity $f$ satisfies
%$$
\begin{equation}\label{3.feasy}
|f'(u)|\le C(1+|u|^{4-\kappa})
\end{equation}
%$$
for some $0<\kappa\leq 4$. In that case, the local existence result of Proposition \ref{Prop2.exist} can be improved as follows.
\begin{proposition}\label{Prop3.subcrit} Let the assumptions of Proposition \ref{Prop2.exist} hold and let, in addition, assumption \eqref{3.feasy} be satisfied. Then, for every $\xi_0\in\Cal E$, there exists $T=T(\|\xi_0\|_{\Cal E})>0$ such that equation \eqref{0.eqmain} possesses a Shatah-Struwe solution $u(t)$ on the interval $t\in[0,T]$ and the following estimate holds:
%$$
\begin{equation}\label{3.main}
\|u\|_{L^4(0,T;L^{12}(\Omega))}\le Q(\|\xi_0\|_{\Cal E})
\end{equation}
%$$
for some monotone function $Q$ which is independent of $u$.
\end{proposition}
\begin{proof} The proof of this statement is similar to the one of Proposition \ref{Prop2.exist}, but we need to check that now the lifespan $T$ depends only on the energy norm of $\xi_0$. To this end, we note that, due to \eqref{3.feasy}, estimate \eqref{3.pnf} can be improved as follows
%$$
\begin{multline}\label{3.nice}
\|P_Nf(v_N+w_N)\|_{L^1(0,t;L^2(\Omega))}\le C\(t+\|v_N\|_{L^{5-\kappa}(0,t;L^{10}(\Omega))}^{5-\kappa}+\|w_N\|_{L^{5-\kappa}(0,t;L^{10}(\Omega))}^{5-\kappa}\)\le\\\le C\(t+t^{\kappa/5}\|v_N\|_{L^{5}(0,t;L^{10}(\Omega))}^{5-\kappa}\)+C\|w_N\|_{L^{5}(0,t;L^{10}(\Omega))}^{5-\kappa}.
\end{multline}
%$$
We see that the first term at the left-hand side of \eqref{3.nice} can be made small  by decreasing $t$ and we need not to make the $L^5(L^{10})$-norm of $v_N$ small. Thus, we may use estimate \eqref{1.distr} and interpolation \eqref{1.int}, to see that
$$
\|v_N\|_{L^5(0,1;L^{10}(\Omega))}\le C\|\xi_0\|_{\Cal E}
$$
and, for every $\eb>0$, we may find $T=T(\eb,\|\xi_0\|_{\Cal E})$ such that
$$
C\(t+t^{\kappa/5}\|v_N\|_{L^{5}(0,t;L^{10}(\Omega))}^{5-\kappa}\)\le \eb,\ \ t\le T(\eb,\|\xi_0\|_{\Cal E}).
$$
Arguing then exactly as in the end of the proof of Proposition \ref{Prop2.exist}, we establish the existence of the desired Shatah-Struwe solution $u$ as well as estimate \eqref{3.main}. Proposition \ref{Prop3.subcrit} is proved.
\end{proof}
As has been already noted in Remark \ref{Rem2.difficult}, the control of the lifespan of the {\it local} Shatah-Struwe
solution in terms of the energy norm together with the control of energy norm due to the energy estimate allows us to extend the local solution for all time and prove the existence of a {\it global} Shatah-Struwe solution $u(t)$ of problem \eqref{0.eqmain}. Namely, the following statement holds.
\begin{corollary}\label{Cor3.main} Let the assumptions of Proposition \ref{Prop3.subcrit} hold and let, in addition, the nonlinearity $f$ satisfy the dissipativity assumption \eqref{2.fdis}. Then, for every $\xi_0\in\Cal E$, there exists a unique global Shatah-Struwe solution $u(t)$ of problem \eqref{0.eqmain} and the following dissipative estimate holds:
%$$
\begin{equation}\label{3.gooddis}
\|\xi_u(t)\|_{\Cal E}+\|u\|_{L^4(t,t+1;L^{12}(\Omega))}\le Q(\|\xi_u(0)\|_{\Cal E})e^{-\alpha t}+Q(\|g\|_{L^2}),
\end{equation}
%$$
where the positive constant $\alpha$ and the monotone function $Q$ are independent of $u$ and $t$.
\end{corollary}
\begin{proof}
Indeed, the uniqueness is proved in Corollary \ref{Cor2.uni} and the dissipative energy estimate is obtained in Corollary \ref{Cor2.dis}. According to this estimate, the energy norm $\|\xi_u(t)\|_{\Cal E}$ cannot blow up in a finite time and, therefore, due to Proposition \ref{Prop3.subcrit}, the local Shatah-Struwe solution $u(t)$ can be extended globally in time. Finally, the dissipative estimate \eqref{3.gooddis} for the Strichartz norm of $u$ follows from \eqref{3.main} and the dissipative estimate \eqref{2.disest} for the energy norm.
\end{proof}
We are now ready to verify the asymptotic compactness of the solution semigroup $S(t)$ of equation \eqref{0.eqmain}
 in the subcritical case. To this end, we split the solution $u$ as follows: $u(t)=v(t)+w(t)$, where $v(t)$ solves the linear problem
 %$$
 \begin{equation}\label{3.linear}
 \Dt^2 v+\gamma\Dt v-\Dx v=0,\ \ \xi_v\big|_{t=0}=\xi_u\big|_{t=0}
 \end{equation}
%$$
and the remainder $w(t)$ satisfies
%$$
 \begin{equation}\label{3.nonlinear}
 \Dt^2 w+\gamma\Dt w-\Dx w=g-f(u),\ \ \xi_w\big|_{t=0}=0.
 \end{equation}
%$$
Then, due to estimate \eqref{1.disstr},
%$$
\begin{equation}\label{3.lindis}
\|\xi_v(t)\|_{\Cal E}+\|v\|_{L^4(t,t+1;L^{12}(\Omega))}\le Q(\|\xi_u(0)\|_{\Cal E})e^{-\alpha t}
\end{equation}
%$$
and, therefore, the $v$-component is exponentially decaying in the energy and Strichartz norms. As the next corollary shows, the $w$-component
is more regular.
\begin{corollary}\label{Cor3.comp} Let the assumptions of Corollary \ref{Cor3.main} hold. Then, there exists $\delta=\delta(\kappa)>0$  and $\delta<1/2$ such that
$$
\xi_w(t)\in \Cal E_\delta:=H_0^{1+\delta}(\Omega)\times H^\delta(\Omega)
$$
and the following estimate holds:
%$$
\begin{equation}\label{3.nonlindis}
\|\xi_w(t)\|_{\Cal E_\delta}+\|w\|_{L^4(t,t+1;W^{\delta,12}(\Omega))}\le Q(\|\xi_u(0)\|_{\Cal E})e^{-\alpha t}+Q(\|g\|_{L^2}),
\end{equation}
%$$
where the monotone function $Q$ and the positive constant $\alpha$ are independent of $u$ and $t$.
\end{corollary}
\begin{proof}Indeed, since the function $G:=(-\Dx)^{-1}g\in H^2$, it only remains to verify estimate \eqref{3.nonlindis} for the function $\bar w(t):=w(t)-G$ which solves
$$
\Dt^2\bar w+\gamma\Dt\bar w-\Dx \bar w=-f(u),\ \ \xi_{\bar w}(0)=\xi_u(0)-(G,0).
$$
Moreover, due to estimate \eqref{1.disstr}, we only need to check that
%$$
\begin{equation}\label{3.est}
\|f(u)\|_{L^1(t,t+1;H^\delta(\Omega))}\le Q(\|\xi_u(0)\|_{\Cal E})e^{-\alpha t}+Q(\|g\|_{L^2}).
\end{equation}
%$$
According to the H\"older inequality and estimate \eqref{3.gooddis},
%$$
\begin{multline*}
\|f(u)\|_{L^1(t,t+1;W^{1,6/5}(\Omega))}\le C(1+\|u^4\Nx u\|_{L^1(t,t+1;L^{6/5}(\Omega))})\le\\\le C(1+\|u\|^4_{L^4(t,t+1;L^{12}(\Omega))}\|\Nx u\|_{L^\infty(t,t+1;L^2(\Omega))})\le Q(\|\xi_u(0)\|_{\Cal E})e^{-\alpha t}+Q(\|g\|_{L^2}).
\end{multline*}
%$$
On the other hand, due to the growth restriction \eqref{3.feasy},
$$
\|f(u)\|_{L^1(t,t+1;L^{10/(5-\kappa)}(\Omega))}\le C(1+\|u\|_{L^5(t,t+1;L^{10}(\Omega))})\le Q(\|\xi_u(0)\|_{\Cal E})e^{-\alpha t}+Q(\|g\|_{L^2}).
$$
The interpolation inequality
$$
\|U\|_{H^\delta}\le C\|U\|_{W^{1,6/5}}^{1-\theta}\|U\|_{L^{10/(5-\kappa)}}^\theta,\ \ \theta=\frac{10}{10+3\kappa},\ \ \delta=\frac{3\kappa}{10+3\kappa}
$$
now gives the desired estimate \eqref{3.est} and finishes the proof of the corollary.
\end{proof}
We conclude our study of the subcritical case by establishing the existence of a global attractor for the associated solution semigroup. For the convenience of the reader, we recall the definition of a global attractor, see \cite{tem,BV,CV} for more details.

\begin{definition}\label{Def3.attr} Let $S(t)$ be a semigroup acting on a Banach space $\Cal E$. Then, a set $\Cal A\subset \Cal E$ is a global attractor of $S(t)$ if
\par
1. The set $\Cal A$ is compact in $\Cal E$;
\par
2. The set $\Cal A$ is strictly invariant: $S(t)\Cal A=\Cal A$;
\par
3. It is an attracting set for the semigroup $S(t)$, i.e., for any bounded set $B\subset\Cal E$ and every neighborhood $\Cal O(\Cal A)$ of the set $\Cal A$, there exists a time $T=T(B,\Cal A)$ such that
$$
S(t)B\subset\Cal O(\Cal A),\ \ \forall t\ge T.
$$
\end{definition}
The next theorem can be considered as the main result of this section.
\begin{theorem}\label{Th3.main} Let the assumptions of Corollary \ref{Cor3.main} hold. Then, the solution semigroup $S(t)$
 associated with problem \eqref{0.eqmain} possesses a global attractor $\Cal A$ in the energy phase space $\Cal E$. Moreover,
 the attractor $\Cal A$ is bounded in more regular space:
 %$$
 \begin{equation}\label{3.attrsmooth}
 \Cal A\in\Cal E_\delta ,\ \ \|\Cal A\|_{\Cal E_\delta}\le C
 \end{equation}
 %$$
 for some $\delta>0$.
\begin{proof} Indeed, according to the abstract attractor existence theorem, we need to verify that $S(t)$ is continuous in $\Cal E$ for every fixed $t$ and that it possesses a compact attracting set in $\Cal E$, see \cite{BV}. The first assertion is satisfied due to Corollary \ref{Cor2.uni} and, according to Corollary \ref{Cor3.comp} and estimate \eqref{3.lindis}, the following set
$$
\Cal B:=\{\xi\in\Cal E_\delta,\ \|\xi\|_{\Cal E_\delta}\le R\}
$$
will be the compact attracting set for $S(t)$ in $\Cal E$ if the radius $R$ is large enough. Thus, all assumptions of the abstract attractor existence theorem are verified and the existence of the attractor $\Cal A$ is proved. It remains to recall that $\Cal A\subset\Cal B$, so \eqref{3.attrsmooth} is also verified and the theorem is proved.
\end{proof}
\begin{remark}\label{Rem3.good} We stated in Theorem \ref{Th3.main} only that the attractor $\Cal A$ is bounded in $\Cal E_\delta$ with $\delta=\frac{3\kappa}{3\kappa+10}$. However, using the standard bootstrapping arguments one can easily show that $\Cal A\in \Cal E_1$ and that its actual regularity is restricted only by the regularity of $\Omega$, $g$ and $f$ (if all the above data is $C^\infty$-smooth, the attractor also will be $C^\infty$-smooth). Moreover, since $H^2\subset C$, the growth rate of $f$ becomes not essential and one can establish the finite-dimensionality of $\Cal A$ exactly as in the well-studied case when $f(u)$ grows slower than $u^3$, see e.g., \cite{BV}. Mention also that the generalization to the non-autonomous case, when, say, the external force $g=g(t)$ depends explicitly on time is also straightforward.
\end{remark}

\end{theorem}

\section{Weak trajectory attractors for critical and supercritical cases}\label{s4}
The aim of this section is to recall the trajectory attractor theory for equation \eqref{0.eqmain} for the case of fast growing nonlinearities developed in \cite{ZelDCDS}. This theory will be essentially used in the next section for proving the dissipativity of Shatah-Struwe solutions in the critical case. Namely, following \cite{ZelDCDS}, we assume that the nonlinearity $f$ satisfies the following conditions:
%$$
\begin{equation}\label{00.3}
\begin{cases}
1.\ \ f\in C^2(\R,\R),  \ \ f(0)=0,\\
2.\ \ |f''(v)|\le C(1+|v|^p),\\
3.\ \ f'(v)\ge -K+\delta|v|^{p+1},
\end{cases}
\end{equation}
%$$
where the exponent $p$ can be  arbitrarily large (of course, we are mainly interested in the case $p\ge3$ since the subcritical case $p<3$ is studied in the previous section). Note that, for the case $p>3$, the energy phase space should be modified:
$$
E:=[H^1_0(\Omega)\cap L^{p+3}(\Omega)]\times L^2(\Omega)
$$
in order to guarantee the finiteness of the energy (since $H^1$ is not embedded into $L^{p+3}$ if $p>3$). We also modify the energy norm for that case as follows:
$$
\|\xi_u(t)\|^2_E:=\|\Dt u(t)\|^2_{L^2}+\|\Nx (t)\|^2_{L^2}+\|u(t)\|^{p+3}_{L^{p+3}}.
$$
In this section, we will work with Galerkin solutions of equation \eqref{0.eqmain}, see Definition \ref{Def1.Gal}.
\begin{proposition}\label{Prop4.Galex} Let the nonlinearity $f$ satisfies assumptions \eqref{00.3} and let $g\in L^2(\Omega)$. Then, for every $\xi_0\in\Cal E$, there exists at least one Galerkin solution $u(t)$, $t\in\R_+$, of problem \eqref{0.eqmain}
\end{proposition}
The assertion of this proposition is standard, so its proof is omitted, see, e.g., \cite{ZelDCDS} for more details.
\par
Our next aim is to state the analogue of the energy inequality for Galerkin solutions. We first note that, arguing in a standard way, one derives the following dissipative energy estimate for the Galerkin approximations $u_N(t)$ (which are the solutions of \eqref{1.galerkin}):
%$$
\begin{equation}\label{4.energyN}
\|\xi_{u_N}(t)\|_{E}^2+\int_t^\infty\|\Dt u_N(\tau)\|^2_{L^2}\,d\tau\le C\|\xi_{u_N}(s)\|^2_Ee^{-\alpha(t-s)}+C(1+\|g\|^2_{L^2}),
\end{equation}
%$$
where $0\le s\le t$ and positive constants $C$ and $\alpha$ are independent of $s$, $t$, $N$ and $u_N$, see \cite{ZelDCDS}.
However, since we do not have the strong convergence $\xi_{u_N}(\tau)\to\xi_{u}(\tau)$ in $E$, we cannot pass to the limit in \eqref{4.energyN} at least in a straightforward way, so we cannot guarantee that \eqref{4.energyN} will remain true for the limit Galerkin solution $u(t)$. To overcome this difficulty, we need to introduce, following again \cite{ZelDCDS}, the so-called $M$-energy functional which generalizes the usual energy functional.

\begin{definition}\label{Def4.M} Let the assumptions of Proposition \ref{Prop4.Galex} hold and let $u$ be a Galerkin solution of problem \eqref{0.eqmain}. We define the
functional $M_u(t)$, $t\ge0$, by the following expression:
%$$
\begin{equation}\label{4.M}
M_u(t):=\inf\bigg\{\liminf_{k\to\infty}\|\xi_{u_{N_k}}(t)\|_{E}\,:\ \
\xi_{u_{N_k}}\rightharpoonup\xi_u, \
\xi_{u_{N_k}}(0)\rightharpoonup\xi_u(0)\bigg\},
\end{equation}
%$$
where the external infinum in the right-hand side of \eqref{4.M}
is taken over all sequences of the Galerkin approximations
 $\{\xi_{u_{N_k}}(t)\}_{k=1}^\infty$ which converge weakly-$*$ in $L^\infty_{loc}(\R_+,E)$ to the
given Galerkin solution~$u$.
\end{definition}
The following corollary gives simple, but important properties of the
 $M$-energy functional.
\begin{corollary}\label{Cor4.M} Let the assumptions of Proposition \ref{Prop4.Galex} hold. Then, for every Galerkin solution $u$ of equation \eqref{0.eqmain}, the following estimates hold:
%$$
\begin{equation}\label{4.34}
M_u(t)<\infty,\ \ \|\xi_u(t)\|_{E}\le M_u(t),
\ \
M_{T_hu}(t)\le M_u(t+h), \ h\ge0,
\end{equation}
%$$
where $(T_hu)(t):=u(t+h)$,
and
%$$
\begin{equation}\label{4.Men}
M_u(t)^2+\int_t^\infty\|\Dt u(t)\|_{L^2}^2\,dt\le
 CM_u(s)^2e^{-\alpha(t-s)}+C(1+\|g\|_{L^2}^2),
\end{equation}
%$$
where $t\ge s\ge0$ and constants $\alpha>0$ and $C>0$ are the same as in
\eqref{4.energyN}.
\end{corollary}
Indeed, estimates \eqref{4.34} are immediate corollaries
of the definitions of the Galerkin solution and the functional  $M_u(t)$ and estimates
 \eqref{4.Men} follow from estimate \eqref{4.energyN} in which we pass to
the limit $N_k\to\infty$.

\begin{remark}\label{Rem4.M} It is well-known (see \cite{BV}) that, in the case $p\le1$,
 we have the strong convergence of Galerkin approximations and, consequently,
%$$
\begin{equation}\label{4.Mgood}
\|\xi_u(t)\|_{E}=M_u(t).
\end{equation}
%$$
So, in this case,  the $M$-energy coincides with the classical one.
Moreover, the same equality will hold in the case when $p\le3$ and $u$ is a Shatah-Struwe solution of problem \eqref{0.eqmain} due to Corollary \ref{Cor2.SSGalerkin}.
But to the best of our knowledge, neither identity
\eqref{4.Mgood}  nor the fact that any solution
$\xi_u\in L^\infty(\R_+,E)$ of \thetag{0.1} can
be obtained as a limit of the
 Galerkin approximations  are known
in the supercritical case $p>3$. Nevertheless, if the solution
$\xi_u(t)$ of problem \eqref{0.eqmain} is  sufficiently regular:
$$
\xi_u\in L^\infty(\R_+,\Cal E_1),\ \
 \Cal E_1:=[H^2(\Omega)\cap H^1_0(\Omega)]\times H^1_0(\Omega),
$$
then it is unique and, arguing as in Corollary \ref{Cor2.SSGalerkin}, we
can show the strong convergence of Galerkin approximations and equality \eqref{4.Mgood}.
\par
It also worth to emphasize that,  in contrast to  the usual energy functional,
the functional $M_u(t)$ is not a priori local with respect to $t$,
i.e. $M_u(T)$ depends not only on $\xi_u(T)$, but also on the whole
trajectory $u$.
\end{remark}
We are now ready to build up a trajectory dynamical system associated with equation \eqref{0.eqmain}, see \cite{CV,CV1,ZelDCDS} for more details.
\begin{definition}\label{Def4.tr} We define the trajectory phase space $K^+$ of
problem \eqref{0.eqmain} as the
 set of all its {\it Galerkin} solutions which correspond to all possible initial data $\xi_0\in E$,  see Definition \ref{Def1.Gal}.
Obviously, $K^+$ is a subset of $L^\infty(\R_+,E)$.
\par
We endow
the trajectory phase space $K^+$ with the topology induced by the
embedding
$$
K^+\subset \Theta^+:=[L^\infty_{loc}(\R_+,E)]^{w^*},
$$
i.e. by the weak-$*$ topology of the space $L^\infty_{loc}(\R_+, E)$.
\par
We also introduce the group of positive time shifts:
%$$
\begin{equation}\label{4.trsem}
T_h:\Theta^+\to\Theta^+,\ h\ge0,\  \ (T_hu)(t):=u(t+h).
\end{equation}
%$$
Then, as not difficult to see,  semigroup \eqref{4.trsem} acts on the trajectory
phase space $K^+$:
%$$
\begin{equation}\label{4.trdyn}
T_h: K^+\to K^+.
\end{equation}
%$$
Semigroup \eqref{4.trdyn} acting on the topological space $K^+$
is called the trajectory dynamical system associated with
equation \eqref{0.eqmain}.
\end{definition}
\begin{remark}\label{Rem4.equiv} As known,  in the  case $p\le1$, the Galerkin solution $u(t)$
of equation \eqref{0.eqmain} is unique and, consequently, this
equation generates a semigroup in the classical energy phase space $\Cal E$
in a standard way:
%$$
\begin{equation}\label{4.30}
S(t): \Cal E\to \Cal E,\ \ t\ge0,\ \ S(t)\xi_u(0):=\xi_u(t).
\end{equation}
%$$
Moreover, in this case, the map
%$$
\begin{equation}\label{4.31}
\Pi_{t=0}:K^+\to \Cal E,\ \ \Pi_{t=0}\xi_u=\xi_u(0),
\end{equation}
%$$
where, by definition, $\Pi_{t=0}u=\xi_u(0)$,
is one to one and realizes a (sequential) homeomorphism between
$K^+$ and $\Cal E^{w}$ (= the space $\Cal E$ endowed by the weak topology). Thus,
%$$
\begin{equation}\label{4.32}
S(t)=\Pi_{t=0}\circ T_t\circ (\Pi_{t=0})^{-1},
\end{equation}
%$$
and, therefore,
the trajectory dynamical system \eqref{4.trdyn} is conjugated
to the classical dynamical system \eqref{4.30} defined on the usual energy
phase space $\Cal E$ endowed with the weak topology.
\par
We note however  that, for fast growing nonlinearities, the uniqueness
problem for \eqref{0.eqmain} is not solved yet (in particular, even in the most interesting for our purposes quintic case $p=3$, the uniqueness of Galerkin solutions is not known) and under the classical approach,  semigroup
\eqref{4.30} can be defined as a semigroup of multivalued maps only. The use of the trajectory dynamical
system \eqref{4.trdyn} allows us  to avoid  multivalued maps
and to apply the standard  attractor theory
in order to study the long time behavior of solutions of~\eqref{0.eqmain}
in the supercritical case.
\end{remark}
As the next step, we intend to define the attractor of the introduced trajectory dynamical system. As usual (see e.g. \cite{CV,CV1,ZelDCDS}), in order to
define the global attractor of the semigroup \eqref{4.trdyn}, we first need to
define the class of bounded sets which will be attracted by
this attractor.
\begin{definition}\label{Def4.bound} A set $B\subset K^+$ is called
$M$-bounded  if the following quantity is finite:
%$$
\begin{equation}\label{4.bound}
\|B\|_M:=\sup_{\xi_u\in B}M_u(0)<\infty.
\end{equation}
%$$
In other words, the set $B\subset K^+$ is $M$-bounded if the modified
energy of all the solutions belonging to $B$ is uniformly bounded.
\end{definition}

\begin{definition}\label{Def4.trattr} A set $\Cal A^{tr}$ is a
global attractor of the trajectory dynamical system \eqref{4.trdyn}
(= the trajectory attractor of equation \eqref{0.eqmain}) if the
following conditions hold:
\par
1.\ \ The set $\Cal A^{tr}$ is a compact $M$-bounded set in $K^+$.
\par
2.\ \ This set is strictly invariant, i.e.  $T_h\Cal A^{tr}=\Cal
  A^{tr}$, for $h\ge0$.
\par
3.\ \ This set is an attracting set for semigroup \eqref{4.trdyn},
i.e. for every $M$-bounded  subset
$B\subset K^+$ and every neighborhood $\Cal O(\Cal A^{tr})$
of $\Cal A^{tr}$ in $K^+$,
 there exists  $T=T(B,\Cal O)$ such that
%$$
\begin{equation}\label{4.39}
T_hB\subset\Cal O(\Cal A^{tr}),\ \ \text{ for }\ \ h\ge T.
\end{equation}
%$$
\end{definition}
The next theorem establishes the existence of the attractor $\Cal A^{tr}$ for the
trajectory dynamical system associated with problem \eqref{0.eqmain}.
\begin{theorem}\label{Th4.trattr} Let the assumptions of Proposition \ref{Prop4.Galex} hold.
Then, semigroup \eqref{4.trdyn} possesses a global attractor $\Cal
A^{tr}$ in the sense of Definition \ref{Def4.trattr} which can be described in the
following way:
%$$
\begin{equation}\label{4.40}
\Cal A^{tr}=\Pi_{t\ge0}\Cal K,\ \ \Pi_{t\ge0}u:=u\big|_{t\ge0}.
\end{equation}
%$$
Here $\Cal K\subset L^\infty(\R, E)$ is the set of all the complete
solutions of problem \eqref{0.eqmain} which are defined
for all $t\in\R$ and can be obtained as a Galerkin limit, i.e.
$\xi_u\in\Cal K$  if and only if  there exist a sequence
of times
$t_k\to-\infty$ and a sequence of solutions $\xi_{u_{N_k}}(t)$
of the problems:
%$$
\begin{equation}\label{4.41}
\begin{cases}
\Dt^2 u_{N_k}+\gamma\Dt u_{N_k}-\Dx u_{N_k}+ P_{N_k} f(u_{N_k})=g_{N_k},
\\
 \xi_{u_{N_k}}(t_k)=\xi^0_k\in E_{N_k}, \ \ t\ge t_k,
\end{cases}
\end{equation}
%$$
where $E_{N_k}:=P_{N_k} E$, such that
%$$
\begin{equation}\label{4.42}
\|\xi_k^0\|_{E}\le C,\ \ \text{ and }\ \
\xi_u=\Theta-\lim_{k\to\infty}\xi_{u_{N_k}},
\end{equation}
%$$
where $C$ is independent of $k$ and
%$$
\begin{equation}\label{4.43}
\Theta:=\bigg[L^\infty_{loc}(\R, E)\bigg]^{w^*}.
\end{equation}
%$$
\end{theorem}
For the proof of this theorem see \cite{ZelDCDS}.
\par
The next standard assertion utilizes the gradient structure of equation \eqref{0.eqmain}.
\begin{corollary}\label{Cor4.grad} Let the assumptions of Theorem \ref{Th4.trattr} hold and
let $\xi_u\in\Cal K$. Then,
%$$
\begin{equation}\label{4.disint}
\int_{-\infty}^{+\infty}\|\Dt u(s)\|_{L^2}^2\,ds\le C(1+\|g\|^2_{L^2}),
\end{equation}
%$$
where the constant $C$ is the same as in \eqref{4.Men},
and moreover, for every $1\ge\beta>0$,
%$$
\begin{equation}\label{4.Dtconv}
\Dt u\in C_b(\R, H^{-\beta}(\Omega)) \ \text{ and } \
\lim_{t\to\pm\infty}\|\Dt u(t)\|_{H^{-\beta}(\Omega)}=0.
\end{equation}
%$$
\end{corollary}
  Indeed, the finiteness of the dissipation integral is an immediate corollary of estimate \eqref{4.Men} and the definition of the set $\Cal K$ and the convergence \eqref{4.Dtconv} follows from this integral and from the embedding $\Theta\subset C_{loc}(R,H^{-\beta}(\Omega))$, see \cite{ZelDCDS} for more details.
\par
The next theorem which establishes the backward regularity of solutions on the trajectory attractor $\Cal A^{tr}$ is crucial for our proof of asymptotic compactness for the quintic case, see the next section.
\begin{theorem}\label{Th4.backsmooth} Let the assumptions of Theorem \ref{Th4.trattr} hold. Then,
for every complete Galerkin solution $\xi_u\in\Cal K$ of equation
\eqref{0.eqmain}, there exists a time $T=T_u$ such that
%$$
\begin{equation}\label{4.bsm}
\xi_u\in C_b((-\infty,T], \Cal E_1)
\end{equation}
%$$
and
%$$
\begin{equation}\label{4.bbb}
\|\xi_u\|_{C_b(-\infty,T; \Cal E_1)}\le C,
\end{equation}
%$$
where the constant $C$ is independent of $u\in\Cal K$.
\end{theorem}
The proof of this theorem is essentially based on the finiteness of the dissipation integral \eqref{4.disint} and is given in \cite{ZelDCDS}, see also \cite{Zel1,Zel2} for the analogous results for the hyperbolic Cahn-Hilliard equations.
\par
To conclude the section, we state a version of the so-called weak-strong uniqueness result which shows that the solution
$\xi_u(t)\in\Cal K$ is unique until it is regular, so the non-uniqueness can appear only after the possible blow up of the strong solution.
\begin{theorem}\label{Th4.uni} Let the assumptions of Theorem \ref{Th4.trattr} hold and
$\xi_u\in\Cal K$ be a complete weak solution of \eqref{0.eqmain}
which satisfies \eqref{4.bsm}, for $t\le T$. We  also assume that
$\xi_v\in\Cal K$ is another complete weak solution which
satisfies
%$$
\begin{equation}\label{4.60}
\xi_u(t)=\xi_v(t),\ \ \text{ for all }\ t\le T'<T.
\end{equation}
%$$
Then, necessarily
$$
\xi_u(t)=\xi_v(t),\ \ \text{ for all }\ t\le T.
$$
\end{theorem}
The proof of this theorem is also given in \cite{ZelDCDS}.

\section{Asymptotic compactness and attractors: the critical case}\label{s5}

In this concluding section, we establish the asymptotic compactness of the Shatah-Struwe solutions and the existence of the global attractor for the solutions semigroup $S(t)$, see \eqref{2.semigroup}, of equation \eqref{0.eqmain} in the critical quintic case. The crucial role in our proof of this fact is played by the trajectory attractor $\Cal A^{tr}$ for the Galerkin solutions of this equation and the backward regularity of solutions on it discussed in the previous section. Namely, combining the results of Section \ref{s2} on the Shatah-Struwe solutions with the trajectory attractor approach for the Galerkin solutions discussed in the previous section, we obtain the following regularity result.

\begin{proposition}\label{Prop5.attr-smooth} Let the nonlinearity $f$ satisfy \eqref{2.extradis} and \eqref{00.3} with $p=3$ and let $g\in L^2(\Omega)$. Then the trajectory attractor $\Cal A^{tr}$ of problem \eqref{0.eqmain} constructed in Theorem \ref{Th4.trattr} is generated by smooth complete solutions of \eqref{0.eqmain}, namely, for any $\xi_u\in \Cal K$,
%$$
\begin{equation}\label{5.smooth}
\xi_u(t)\in\Cal E_1,
\end{equation}
%$$
for all $t\in\R$.
\end{proposition}
\begin{proof} Indeed, due to Theorem \ref{Th4.backsmooth}, we know that $\xi_u(t)\in\Cal E_1$ for all $t\le T$. Moreover, due to Theorem \ref{Th2.SSgwp} there is an extension $\bar u(t)$ for $t\ge T$ such that $\bar u(t)=u(t)$ for $t\le T$ and $\bar u(t)$ is a Shatah-Struwe solution of equation \eqref{0.eqmain} for all $t\in\R$. Then, due to Proposition \ref{Prop2.smooth} and the fact that $\xi_{\bar u}(T)\in\Cal E_1$, we conclude that $\xi_{\bar u}(t)\in\Cal E_1$ for all $t\in\R$.
\par
Furthermore, due to Corollary \ref{Cor2.SSGalerkin}, any Shatah-Struwe solution is a Galerkin solution as well and the Galerkin approximations converge even {\it strongly} in $\Cal E$ to that solution. By this reason, the modified energy coincides with the usual one (i.e., identity \eqref{4.Mgood} holds) for any Shatah-Struwe solution and this, together with the definition of the set $\Cal K$ implies that $\bar u\in\Cal K$.
Finally, due to the uniqueness Theorem \ref{Th4.uni}, $u(t)=\bar u(t)$ for all $t\in\R$. This gives \eqref{5.smooth} and finishes the proof of the proposition.
\end{proof}
\begin{remark}\label{Rem5.energy} Note that, at this stage we have established only the
 $\Cal E_1$-regularity of any solution $u\in\Cal K$ and the global boundedness of $\Cal K$ in $C_b(\R,\Cal E)$ (due to the energy estimate).
 However, since we do not control the growth rate of the Strichartz norm with respect to $T$ in estimate \eqref{2.badSS}, we still do not have boundedness of $\xi_u(t)$ as $t\to\infty$ in the $\Cal E_1$-norm. Nevertheless, we obviously have the energy {\it equality} for every $u\in\Cal K$. This, together with the standard energy method will allow us to establish the asymptotic compactness which a posteriori will give us the desired control of the Strichartz norm and, finally, we will verify that $\Cal K$ is bounded in $C_b(\R,\Cal E_1)$ as well.
\end{remark}
The next theorem can be considered as the main result of this section.

\begin{theorem}\label{Th5.main} Let the assumptions of Proposition \ref{Prop5.attr-smooth} hold. Then, the solution semigroup
 $S(t)$ generated by the Shatah-Struwe solutions of equation \eqref{0.eqmain} possesses a global attractor $\Cal A$ in the space $\Cal E$ (see Definition \ref{Def3.attr}) which is a subset of  $\Cal E_1$. Moreover,
 %$$
 \begin{equation}\label{5.gl-tr}
 \Cal A=\Pi_{t=0}\Cal A^{tr},
 \end{equation}
 %$$
 where $\Cal A^{tr}$ is a trajectory attractor of equation \eqref{0.eqmain} constructed in Theorem \ref{Th4.trattr} (based on the Galerkin solutions of equation \eqref{0.eqmain}).
\end{theorem}
\begin{proof} Indeed, due to estimate \eqref{2.disSS}, the ball
$$
B:=\{\xi\in\Cal E,\ \|\xi\|_{\Cal E}\le R\}
$$
is an absorbing ball for the semigroup $S(t)$ in $\Cal E$ and, arguing as in Corollary \ref{Cor2.uni}, we see that the semigroup $S(t)$ is
continuous in $\Cal E$ for every fixed $t$. Thus, according to the abstract attractor existence theorem (see \cite{BV,hale,MirZel}), we only need to verify the asymptotic compactness of the semigroup $S(t)$. Namely, we need to check that, for every sequence $\xi_n\in B$ and every sequence of times $t_n\to\infty$, the sequence
$S(t_n)\xi_n$ is precompact in $\Cal E$, i.e., that there exists a subsequence $n_k$ such that
%$$
\begin{equation}\label{5.strong}
S(t_{n_k})\xi_{n_k}\to\xi_\infty
\end{equation}
%$$
strongly in $\Cal E$.
\par
To prove the strong convergence we will utilize the so-called energy method, see e.g., \cite{rosa,ball}. We start with the elementary observation that, without loss of generality, we may assume that $S(t_n)\xi_n\to\xi_\infty$ {\it weakly} in $\Cal E$. This follows from the fact that the sequence  $S(t_n)\xi_n$ is bounded due to energy estimate and the Banach-Alaoglu theorem. Let us denote by $v_n(t):=S(t)\xi_n$ the corresponding Shatah-Struwe solutions of equation \eqref{0.eqmain} and fix $u_n:=T_{t_n}v_n$. Then, $u_n(t)$ are also Shatah-Struwe solutions of equation \eqref{0.eqmain} defined on time interval $t\in[-t_n,\infty)$ and
$$
\xi_{u_n}(0)\rightharpoonup\xi_\infty
$$
weakly in $\Cal E$. Since every Shatah-Struwe solution is a Galerkin solution and the $M$-energy of them coincide with the usual energy, by the definition of the trajectory attractor $\Cal A^{tr}$, we may assume without loss of generality that $\xi_{u_n}$ is weakly-star  convergent to some
Galerkin solution $\xi_u(t)\in\Cal A^{tr}$ in the space $L^\infty_{loc}(\R_+,\Cal E)$. Moreover, if we extend the functions $\xi_{u_n}(t)$, say, by zero for $t\le-t_n$, we also may assume that
$$
\xi_{u_n}(t)\to\xi_u(t), \ \ \text{ weakly star in } L^\infty_{loc}(\R,\Cal E)
$$
and that $\xi_u\in\Cal K$ with $\xi_u(0)=\xi_\infty$, see \cite{ZelDCDS} for the details.
\par
Multiplying now equations \eqref{0.eqmain} for $u_n$ by $\Dt u_n+\alpha u_n$, where $\alpha>0$ will be fixed below,
we end up with the following energy type identity:
%$$
\begin{equation}\label{5.ena}
\frac d{dt}E_\alpha(u_n)+\kappa E_\alpha(u_n)+G_\alpha(u_n)+(\Phi_\alpha(u_n),1)+(g_\alpha, u_n)=0,
\end{equation}
%$$
where $\kappa>0$ is a parameter, $g_\alpha=(\kappa-\alpha)g$, $\Phi_\alpha(u):=\alpha f(u)u-\kappa F(u)$,
$$
E_\alpha(u):=\frac12\|\xi_u\|^2_{\Cal E}+(F(u),1)-(g,u)+\alpha(u,\Dt u)+\frac12\alpha\gamma\|u\|^2_{L^2}
$$
and
$$
G_\alpha(u):=(\gamma-\alpha-\frac\kappa2)\|\Dt u\|^2_{L^2}+\(\alpha-\frac\kappa2\)\|\Nx u\|^2_{L^2}-\kappa\alpha(u,\Dt u)-\frac{\gamma\alpha\kappa}2\|u\|^2_{L^2}.
$$
We recall that the above calculations are justified since any Shatah-Struwe solution satisfies the energy equality. We now fix the positive constants $\alpha$ and $\kappa$ to be small enough that the quadratic form $G_\alpha$ is positive definite:
$$
K_1\|\xi_u\|^2_{\Cal E}\le G_\alpha(u)\le K_2\|\xi_u\|^2_{\Cal E}
$$
for some positive $K_1$ and $K_2$. We also assume that $4\kappa\le \alpha$ which guarantees that
$$
\Phi_\alpha(u)\ge-C,
$$
due to assumption \eqref{2.extradis}. Integrating now equality \eqref{5.ena} with respect to $t\in[-t_n,0]$, we arrive at
%$$
\begin{equation}\label{5.en-n}
E_\alpha(u_n(0))+\int_{-t_n}^0e^{\kappa s}\(G_\alpha(u_n(s))+(\Phi_\alpha(u_n(s)),1)+(g_\alpha,u_n(s))\)\,ds=E_\alpha(\xi_n)e^{-\kappa t_n}.
\end{equation}
%$$
 We want now  to pass to the limit $n\to\infty$ in equality \eqref{5.en-n}. To this end, we remind that $\xi_{u_n}$ is uniformly bounded in $L^\infty(\R_-,\Cal E)$ and is weakly-star convergent in this space to the solution $\xi_u\in\Cal K$. Moreover, we also know that $\xi_n(0)\to\xi_\infty=\xi_u(0)$ weakly in $\Cal E$. Using the compactness of the embedding $C_{loc}(\R_-,\Cal E)\subset C_{loc}(\R_-,L^2(\Omega))$, we also conclude that $u_n\to u$ {\it strongly} in $C_{loc}(\R_-,L^2(\Omega))$ and, in particular, almost everywhere. Therefore, since $\Phi_\alpha(u)$ is bounded from below and the quadratic form $G_\alpha(u)$ is positive definite, using also the Fatou lemma, we conclude that
%$$
\begin{multline*}
\liminf_{n\to\infty}\int_{-t_n}^0e^{\kappa s}\(G_\alpha(u_n(s))+(\Phi_\alpha(u_n(s)),1)+(g_\alpha,u_n(s))\)\,ds\ge\\\ge \int_{-\infty}^0e^{\kappa s}\(G_\alpha(u(s))+(\Phi_\alpha(u(s)),1)+(g_\alpha,u(s))\)\,ds
\end{multline*}
%$$
and, analogously,
%$$
\begin{equation}\label{5.wrong}
\liminf_{n\to\infty}E_\alpha(u_n(0))\ge E_\alpha(u(0)).
\end{equation}
%$$
Thus, taking into the account that $\xi_n$ is uniformly bounded in $\Cal E$, we end up with
%$$
\begin{equation}\label{5.en-lim}
E_\alpha(u(0))+\int_{-\infty}^0e^{\kappa s}\(G_\alpha(u(s))+(\Phi_\alpha(u(s)),1)+(g_\alpha,u(s))\)\,ds\le0.
\end{equation}
%$$
We now recall that $u\in\Cal K$, so, by Proposition \ref{Prop5.attr-smooth}, $u$ is smooth and, therefore, it satisfies the energy {\it equality}. Thus, repeating the derivation of \eqref{5.en-n}, but for the function $u$, we see that the last inequality is actually the equality. This is possible only if \eqref{5.wrong} is actually {\it equality}. Using now that, due to the Fatou lemma
$$
\liminf_{n\to\infty}(F(u_n(0)),1)\ge (F(u(0)),1)\ \text{and}\ \liminf_{n\to\infty}\|\xi_{u_n}(0)\|^2_{\Cal E}\ge\|\xi_u(0)\|^2_{\Cal E},
$$
we see that
$$
\|\xi_u(0)\|^2_{\Cal E}=\liminf_{n\to\infty}\|\xi_{u_n}(0)\|^2_{\Cal E}.
$$
Thus, since $\xi_{u_n}(0)\rightharpoonup\xi_{u}(0)$, we may assume without loss of generality that
$$
S(t_n)\xi_n=\xi_{u_n}(0)\to\xi_\infty=\xi_u(0)
$$
strongly in $\Cal E$. This proves the desired asymptotic compactness of the semigroup $S(t)$.
\par
Thus, by the abstract attractor existence theorem, there exists a global attractor $\Cal A$ for the semigroup $S(t)$ associated with equation \eqref{0.eqmain} and, obviously,
$$
\Cal A\subset\Pi_{t=0}\Cal K.
$$
The opposite inclusion follows from the fact that $\Cal K$ consists of smooth solutions which are the Shatah-Struwe ones. So, the equality \eqref{5.gl-tr} is also proved  and the theorem is proved.
\end{proof}
We now want to verify that the constructed attractor is {\it bounded} in $\Cal E_1$. To this end, we need the following result.

\begin{corollary}\label{Cor5.stri} Let the assumptions of Theorem \ref{Th5.main} hold. Then, the restriction of the trajectory set
$\Cal K$ to the time interval $t\in[0,1]$ is a compact set of $L^4(0,1;L^{12}(\Omega))$:
$$
\Cal K\big|_{t\in[0,1]}\subset\subset L^4(0,1;L^{12}(\Omega)).
$$
\end{corollary}
\begin{proof} Indeed, due to Theorem \ref{Th5.main}, the attractor  $\Cal A$ is compact in $\Cal E$. Then, arguing as in the proof of Proposition \ref{Prop2.exist}, we see that, for every $\eb>0$, there exists $T=T(\eb)$ such that, for any Shatah-Struwe solution $u(t)$ starting from the attractor ($\xi_u(0)\in\Cal A$), we have
$$
\|u\|_{L^4(0,T(\eb);L^{12}(\Omega))}\le\eb
$$
or, in other words,
%$$
\begin{equation}\label{5.Kstr}
\|\Cal K\big|_{t\in[0,T(\eb)]}\|_{L^4(0,T(\eb);L^{12}(\Omega))}\le\eb.
\end{equation}
%$$
Since the set $\Cal K$ is invariant with respect to time shifts ($T_h\Cal K=\Cal K$), we have proved that, for any $u\in\Cal K$
%$$
\begin{equation}\label{5.str}
\sup_{T\in\R}\|u\|_{L^4(T,T+1;L^{12}(\Omega))}\le C,
\end{equation}
%$$
where the constant $C$ is independent of $u$.
\par
Since $\Cal A$ is compact, verifying the continuity of the solution map $S:\xi_u(0)\to u$
 as the map from $\Cal A$ to $L^4(0,1;L^{12}(\Omega))$ will prove the corollary. To this end, we first observe that using the uniform estimate \eqref{5.Kstr} and arguing as in the proof of Corollary \ref{Cor2.uni}, we see that
 %$$
 \begin{equation}\label{5.lip}
 \|\xi_{u_1}(t)-\xi_{u_2}(t)\|_{\Cal E}\le Ce^{Kt}\|\xi_{u_1}(0)-\xi_{u_2}(0)\|_{\Cal E},
 \end{equation}
 %$$
 where $C$ and $K$ are independent of $\xi_{u_i}(0)\in\Cal A$. Thus, the map $S$ is continuous as the map from $\Cal E$ to $C(0,1;\Cal E)$.
\par
To prove the continuity in the Strichartz norm, we note that analogously to \eqref{2.SSdif},
$$
\|f(u_1(t))-f(u_2(t))\|_{L^2(\Omega)}\le C(1+\|u_1(t)\|_{L^{12}(\Omega)}^4+\|u_2(t)\|^4_{L^{12}(\Omega)})\|\xi_{u_1}(t)-\xi_{u_2}(t)\|_{\Cal E}.
$$
This estimate, together with \eqref{5.lip} and \eqref{5.str}, gives
$$
\|f(u_1)-f(u_2)\|_{L^1(0,1;L^2(\Omega))}\le C\|\xi_{u_1}(0)-\xi_{u_2}(0)\|_{\Cal E},
$$
where the constant $C$ is independent of $\xi_{u_i}(0)\in\Cal A$. Applying now the Strichartz estimate \eqref{1.linstr} to
equation \eqref{2.dif}, we get
$$
\|u_1-u_2\|_{L^4(0,1;L^{12}(\Omega))}\le C\|\xi_{u_1}(0)-\xi_{u_2}(0)\|_{\Cal E}.
$$
Thus, the map $S$ is indeed continuous as a map from $\Cal E$ to $L^4(0,1;L^{12}(\Omega))$ and the corollary is proved.
\end{proof}
We are finally ready to state the result on the boundedness of the global attractor in $\Cal E_1$.

\begin{theorem} Let the assumptions of Theorem \ref{Th5.main} hold. Then the global attractor $\Cal A$ of the solution semigroup $S(t)$ associated with equation \eqref{0.eqmain} is a bounded set in $\Cal E_1$.
\end{theorem}
\begin{proof}Indeed, due to estimate \eqref{4.bbb} for any complete solution $u\in\Cal K$ and due to the invariance of $\Cal K$, it is sufficient to verify the following estimate:
%$$
\begin{equation}\label{5.e1}
\|\xi_u(t)\|_{\Cal E_1}\le Q(\|\xi_u(0)\|_{\Cal E_1}),\ \ t\ge0,
\end{equation}
%$$
where the monotone function $Q$ is independent of $t\ge0$ and $\xi_u(0)\in\Cal A$.
\par
We will proceed analogously to the proof of Proposition \ref{Prop2.smooth}, but will improve estimate \eqref{2.03} using the information on the compactness of $\Cal K$ in the Strichartz norm. Namely, due to that compactness and estimate \eqref{5.str}, for every $\eb>0$, we can split the solution  $u$ in a sum $u(t)=\bar u(t)+\tilde u(t)$, where
%$$
\begin{equation}\label{5.small}
\sup_{T\ge0}\|\tilde u\|_{L^4(T,T+1;L^{12}(\Omega))}\le \eb
\end{equation}
%$$
and the other function is smooth:
%$$
\begin{equation}\label{5.smooth1}
\|\bar u(t)\|_{\Cal E_1}\le C_\eb,\ \ t\ge0,
\end{equation}
%$$
where the constant $C_\eb$ depends on $\eb$, but is independent of $t$ and $u\in\Cal K$. Using this decomposition, we improve \eqref{2.03} as follows
%$$
\begin{multline}
|(f'(u)v,\Dt v)|\le (|f'(\bar u+\tilde u)-f(\bar u)|,|v|\cdot|\Dt v|)+(|f'(\bar u)|,|v|\cdot|\Dt v|)\le\\\le C((1+|\bar u|^3+|\tilde u|^3)|\tilde u|,|v|\cdot|\Dt v|)+C\|f'(\bar u)\|_{L^\infty}\|v\|_{L^2}\|\Dt v\|_{L^2}\le \\\le C(1+\|\tilde u\|^3_{L^{12}}+\| u\|^3_{L^{12}})\|\tilde u\|_{L^{12}}\|\xi_v\|^2_{\Cal E}+\eb\|\xi_v\|^2_{\Cal E}+C_\eb\|\Dt u\|^2_{L^2}=l_\eb(t)\|\xi_v\|^2_{\Cal E}+C_\eb\|\xi_u\|^2_{\Cal E},
\end{multline}
%$$
where $l_\eb(t):=\eb+C\(1+\|\tilde u\|^3_{L^{12}}+\| u\|^3_{L^{12}}\)\|\tilde u\|_{L^{12}})$. Then, due to \eqref{5.str} and \eqref{5.small}, we have
%$$
\begin{equation}\label{5.l}
\int_t^{t+1}l_\eb(t)\,dt\le C\eb,
\end{equation}
%$$
where the constant $C$ is independent of $\eb$ and on $u\in\Cal K$. Inserting this estimate into \eqref{2.02}, we have
%$$
\begin{equation}\label{5.better}
\frac12\frac d{dt}\|\xi_v(t)\|^2_{\Cal E}+\gamma\|\Dt v(t)\|^2_{L^2}\le l_\eb(t)\|\xi_v(t)\|^2_{\Cal E}+C_\eb\|\xi_u(t)\|^2_{\Cal E}.
\end{equation}
%$$
Multiplying now equation \eqref{2.diffeq} by $\alpha v$, where $\alpha>0$ is a small parameter, integrating over $\Omega$ and using \eqref{2.technical}, we derive
$$
\frac d{dt}(\alpha(v(t),\Dt v(t))+\frac12\alpha\gamma\|v(t)\|^2_{L^2})+\alpha\|\Nx v(t)\|^2_{L^2}\le K\alpha\|\xi_u(t)\|^2_{L^2}+\alpha\|\Dt v\|^2.
$$
Taking a sum of this inequality with \eqref{5.better} and fixing $\alpha>0$ to be small enough, we finally arrive at
$$
\frac d{dt}\(\frac12\|\xi_v(t)\|^2_{\Cal E}+\alpha(v(t),\Dt v(t))+\frac12\alpha\gamma\|v(t)\|^2_{L^2}\)+(\kappa-l_\eb(t))\|\xi_v(t)\|^2_{\Cal E}\le K\alpha\|\xi_u(t)\|^2_{\Cal E}
$$
for some positive constant $\kappa$ which is independent of $\eb$ and $u$. Fixing now $\eb>0$ to be small enough,
applying the Gronwall inequality and estimating the term containing $l_\eb(t)$ using \eqref{5.l}, we get
$$
\|\xi_v(t)\|^2_{\Cal E}\le Ce^{-\kappa t}\|\xi_v(0)\|^2_{\Cal E}+C\|\xi_u\|^2_{C(\R_+,\Cal E)}\le C\(\|\xi_v(0)\|^2_{\Cal E}+1\).
$$
Estimate \eqref{2.maximal} gives now the desired estimate \eqref{5.e1} and finishes the proof of the theorem.
\end{proof}
\begin{remark}\label{Rem5.problems} Since $H^2\subset C$ in the 3D case,
the proved boundedness of the global attractor $\Cal A$ in the space $\Cal E_1$ allows us to verify the further regularity of the attractor by straightforward bootstrapping, so, similarly to the subcritical case, the actual regularity of the attractor is restricted by the regularity of $f$ and $g$ only. Moreover, the finite-dimensionality of $\Cal A$ can be obtained also exactly as in the subcritical case.
\par
However, we emphasize that, in contrast to the subcritical case, our proof of the existence of the global attractor $\Cal A$ and its further regularity is strongly based on the gradient structure of equation \eqref{0.eqmain} and the finiteness of the dissipation integral \eqref{4.disint}.
Thus, the extension of the results of this section to the case of non-autonomous external forces $g=g(t)$ or to systems of equations of the form \eqref{0.eqmain} with non-gradient nonlinearity $f$ is still an open problem. As we have already mentioned,
the key difficulty in this problem is to establish the dissipative estimate for the Strichartz norm of any Shatah-Struwe solution $u$ of the form
%$$
\begin{equation}\label{5.key}
\|u\|_{L^4(T,T+1;L^{12}(\Omega))}\le Q(\|\xi_u(0)\|_{\Cal E})e^{-\alpha T}+Q(\|g\|_{L^2}).
\end{equation}
%$$
This estimate cannot be obtained directly from the proof of Theorem \ref{Th2.SSgwp} and we do not know whether or not it is actually true even in the autonomous case considered in this section. Nevertheless, we conjecture that it is true at least in the autonomous case since, a posteriori, based on the existence of the compact global attractor $\Cal A$, on can verify a slightly weaker version of \eqref{5.key}, namely, that for every bounded set $B\subset\Cal E$, there exists $T=T(B)$ such that
$$
\|u\|_{L^4(t,t+1;L^{12}(\Omega))}\le Q(\|g\|_{L^2}),\ \ t\ge T,
$$
so one only needs to verify \eqref{5.key} on a finite time interval and we expect that it can be done using the concentration compactness arguments, see e.g. \cite{tao}. On the other hand, up to the moment, we do not know how to verify \eqref{5.key} in the non-autonomous case.
\end{remark}

\end{document}